\documentclass[a4paper,UKenglish,cleveref, autoref, thm-restate]{lipics-v2021}

\nolinenumbers

\hideLIPIcs

\usepackage[dvipsnames]{xcolor}
\usepackage{graphicx} 
\usepackage{float} 
\usepackage{cite}

\newtheorem{fact}{Fact}

\newcommand{\boundary}{\partial}
\newcommand{\R}{\mathbb{R}}
\newcommand{\T}[1]{\widetilde{#1}}
\newcommand{\eps}{\varepsilon}
\newcommand{\EMPH}[1]{\textit{\textbf{#1}}}
\DeclareMathOperator{\im}{im}

\DeclareMathOperator{\spn}{span}

\usepackage{thmtools}

\title{Hodge Decomposition and General Laplacian Solvers for Embedded Simplicial Complexes}

\author{Mitchell Black}{School of Electrical Engineering and Computer Science, Oregon State University, USA}{blackmit@oregonstate.edu}{}{}

\author{Amir Nayyeri}{School of Electrical Engineering and Computer Science, Oregon State University, USA}{nayyeria@eecs.oregonstate.edu }{}{}

\funding{The authors were supported in part by NSF grants CCF-1941086 and CCF-1816442.}

\acknowledgements{The authors would like to thank the reviewers for their helpful comments, especially for an observation that improved the dependence on $\beta$ in the runtime.}

\authorrunning{M.~Black and A.~Nayyeri} 
\titlerunning{Hodge Decomposition and General Laplacian Solvers}
\Copyright{Mitchell Black and Amir Nayyeri} 



\relatedversion{} 

\EventEditors{Miko{\l}aj Boja\'{n}czyk, Emanuela Merelli, and David P. Woodruff}
\EventNoEds{3}
\EventLongTitle{49th International Colloquium on Automata, Languages, and Programming (ICALP 2022)}
\EventShortTitle{ICALP 2022}
\EventAcronym{ICALP}
\EventYear{2022}
\EventDate{July 4--8, 2022}
\EventLocation{Paris, France}
\EventLogo{}
\SeriesVolume{229}
\ArticleNo{90}

\begin{document}

\maketitle

\begin{abstract}
    We describe a nearly-linear time algorithm to solve the linear system $L_1x = b$ parameterized by the first Betti number of the complex, where $L_1$ is the 1-Laplacian of a simplicial complex $K$ that is a subcomplex of a collapsible complex $X$ linearly embedded in $\R^{3}$. Our algorithm generalizes the work of Black et al.~[SODA2022] that solved the same problem but required that $K$ have trivial first homology.
    Our algorithm works for complexes $K$ with arbitrary first homology with running time that is nearly-linear with respect to the size of the complex and polynomial with respect to the first Betti number.
    The key to our solver is a new algorithm for computing the Hodge decomposition of 1-chains of $K$ in nearly-linear time. Additionally, our algorithm implies a nearly quadratic solver and nearly quadratic Hodge decomposition for the 1-Laplacian of any simplicial complex $K$ embedded in $\R^{3}$, as $K$ can always be expanded to a collapsible embedded complex of quadratic complexity.
\end{abstract}

\section{Introduction}
The $d$th combinatorial Laplacian of a simplicial complex $K$ is a linear operator that acts on vectors of real numbers associated to the $d$-simplices of $K$. The $d$th combinatorial Laplacian is defined as
\[
L_d = \boundary_{d}^{T}\boundary_d + \boundary_{d+1}\boundary_{d+1}^{T},
\]
where $\boundary_d:C_d(K)\to C_{d-1}(K)$ is the $d$th boundary map of $K$, and $C_d(K)$ is the $d$th chain group of $K$.  The $d$th Laplacian encodes the incidence of $(d-1)$-, $d$- and $(d+1)$-simplices.  In particular, the $0$th Laplacian $L_0$ is composed of a constant map $\partial_0^T\partial_0$, and the well-known graph Laplacian $\partial_1\partial_1^T$.  The graph Laplacian matrix and its algebraic properties have been extensively studied in algebraic and spectral graph theory, a topic that has flourished into a rich field with many applications in computer science such as graph clustering \cite{shi_cuts,ng_clustering}, graph sparsification \cite{spielman_sparsification}, and max flow solvers \cite{Christiano2011MaxFlow} (see Spielman's book and references therein~\cite{SpielmanSpectralBook}).
\par
A highlight of recent advances in algorithmic spectral graph theory is nearly-linear time solvers for linear systems on the graph Laplacian that emerged as a result of decades of research~\cite{SpielmanTengSolver, KoutisMP10, KoutisMP11, Bern05SuppGraphPrec, Boman03SuppTheoPrec, Kelner13SimpleSDD, CohenLogSqrtSolver14, jambulapati2020ultrasparse, Vaidya1990}.  These results imply nearly-linear time solvers for the more general class of symmetric diagonally dominant matrices.  They also have triggered research to find out which classes of linear systems admit nearly-linear time solvers~\cite{Kyng2020HardnessResultsLinea}. Moreover, these solvers are used for different application areas such as approximation algorithm design and numerical analysis~\cite{Christiano2011MaxFlow,Boman08FastEllpiticPDE}.
\par
Recent work has attempted to extend the success of graph Laplacian solvers to higher dimensional Laplacians. Cohen et al.~initiated this line of work by introducing a nearly-linear solver for the 1-Laplacian of collapsible complexes embedded in $\R^{3}$ \cite{OneLaplaciansCohen14}. Black et al.~continued this work by considering complexes with trivial first homology that were subcomplexes of collapsible complexes embedded in $\R^{3}$ \cite{CollapsibleUniverseBlack22}. The solver of Black et al.~implies a nearly quadratic solver for any complex with trivial first homology embedded in $\R^{3}$; they show that a complex embedded in $\R^{3}$ can be extended to a collapsible embedded complex with at most quadratic complexity.
\par
In this paper, we extend the work of Cohen et al.~and Black et al.~to any subcomplex of a collapsible complex embedded in $\R^{3}$, regardless of the rank of its first homology group.
The running time of our solver is nearly-linear with respect to the size of the collapsible complex, and polynomial with respect to the rank of its first homology group.
The main tool in our paper is a new algorithm for computing the Hodge Decomposition of a $1$-chain.
\par
Computing the Hodge decomposition is a problem of independent interest since the
Hodge decomposition has found a myriad of applications in topological data analysis, numerical analysis, and computer graphics among other areas~\cite{Jiang11RankingWithHodge, Candogan11HodgeGame, Arnold2010ExtCalcHodge, deSilva2011PersCircCoord, Tong2003DiscreteVecFieldDecomp, Xu2012HodgeRank, Tahbaz2010DistCovVerSensor, Friedman96BettiNumbers, Crane2013DGP}.
The Hodge decomposition can be computed exactly in $O(n^\omega)$ time by solving a constant number of systems of linear equations, where $\omega$ is the matrix multiplication constant.
(Approximately) computing the Hodge decomposition in nearly-linear time has been an open question with many possible applications.
\par
Cohen et al.~describe nearly-linear projection operators into the coboundary space and cycle space, which implies Hodge decomposition for complexes with trivial homology as the boundary and cycle spaces are identical in this case.
In this paper, we describe projection operators into the boundary and harmonic spaces for an arbitrary subcomplex of a collapsible simplicial complex embedded in $\R^3$. Our boundary projection operator is key to our solver.  Our results imply 1-Laplacian solvers and projection operators for any simplicial complex embedded in $\R^3$ that are quadratic in the size of the complex and polynomial in the first Betti number; these follow from the fact that any complex in $\R^{3}$ can be extended to a collapsible complex in $\R^{3}$ with a quadratic number of simplices \cite[Corollary 3.3]{CollapsibleUniverseBlack22}.
\par
While this paper presents a positive result on extending graph Laplacian solvers to a more general class of Laplacians, a recent work by Ding et al.~\cite{Ding22LaplacianHardness} shows that solving linear equations in arbitrary 1-Laplacians (and therefore arbitrary $d$-Laplacians) is as hard as solving arbitrary sparse linear equations with bounded integer entries and bounded condition number. An interesting open question is whether or not there exist fast solvers for other classes of simplicial complexes.

\subsection{Our Results}
Let $X$ be a collapsible simplicial complex with a known collapsing sequence embedded in $\R^3$, and let $K$ be a subcomplex of $X$.  The first result of this paper is a $1$-Laplacian solver for $K$.  Recall that $L_1 = \partial_2\partial_2^T + \partial_1^T\partial_1$. We define $L_1^{up} = \partial_2\partial_2^T$ and $L_1^{down} = \partial_1^T\partial_1$.  We refer to $L_1^{up}$ and $L_1^{down}$ as the up-Laplacian and down-Laplacian, respectively.

\begin{restatable}{theorem}{solverthm}
\label{thm:main_solver_thm}
Let $X$ be a collapsible simplicial complex with a known collapsing sequence linearly embedded in $\R^3$, and let $K \subset X$ be a subcomplex of $X$.
For any $\varepsilon>0$, there is an operator $LaplacianSolver(X, K, \varepsilon)$ such that
\[
(1-\varepsilon)(L_1[K])^+ \preceq LaplacianSolver(X, K, \varepsilon) \preceq (L_1[K])^+.
\]
where $(L_1[K])^+$ is the pseudoinverse of the 1-Laplacian $L_1[K]$. Further, for any $x\in C_1$,
$LaplacianSolver(X, K, \varepsilon)\cdot x$ can be computed in
$\T O\left(
    \beta^3 \cdot n \cdot \log n \cdot \log({n}/{(\lambda_{\min}(K)\cdot \lambda_{\min}(X)\cdot\eps)})
\right)\footnotemark$
time, where $n$ is the total number of simplices in $X$, $\lambda_{\min}(K)$ and $\lambda_{\min}(X)$ are the smallest nonzero eigenvalues of $L^{up}_1(K)$ and $L^{up}_1(X)$ respectively, and $\beta$ is the rank of the first homology group of $K$.
\end{restatable}

\footnotetext{The $\tilde{O}(\cdot)$ notations hides a factor of $\log \log n$.}

This result is a generalization of Theorem 1.1 of Black et al.~\cite{CollapsibleUniverseBlack22} that requires $K$ to have trivial first homology.
Their running time depends on $\log (n\kappa/\eps)$, with $\kappa$ being the condition number of $L_1^{up}(K)$ within the boundary space. The condition number is defined $\kappa = \lambda_{\max}(K)/\lambda_{\min}(K)$, where $\lambda_{\max}(K)$ is the largest eigenvalue of $L_1^{up}(K)$, and $\lambda_{\min}(K)$ is the smallest \emph{nonzero} eigenvalue of $L_1^{up}(K)$.  We observe that $\lambda_{\max}(K)$ is polynomially bounded with respect to the size of the complex (Lemma \ref{lem:bounded_lambda_max}.)  Therefore, the log dependence of the running time of Black et al.'s solver can be simplified to $\log (n/(\lambda_{\min}(K)\cdot\eps))$.  The running time of Theorem~\ref{thm:main_solver_thm}, in contrast, has an extra dependence to $\lambda_{\min}(X)$ within the log, in addition to a polynomial dependence to $\beta$. For the special case that $\beta = 0$, we can eliminate the dependence on $\lambda_{\min}(X)$ with a more careful analysis and match the running time of Black et al.

\subparagraph{} The new ingredient that makes Theorem~\ref{thm:main_solver_thm} possible is an approximate projection operator onto the boundary space.
Lacking this operator, the previous papers had to assume that $K$ has trivial homology and use a projection into the cycle space instead.

\begin{restatable}{lemma}{boundaryprojection}
\label{lem:boundary_form_cbd_harm}
Let $K$ be a simplicial complex linearly embedded in a collapsible complex $X$ with a known collapsing sequence that is embedded in $\R^3$, and let $\Pi_{bd}$ be the orthogonal projection operator into the space of boundary $1$-chains in $K$. For any $\eps>0$, there is an operator $\widetilde\Pi_{bd}(\eps)$, such that
\[
    (1-\varepsilon)\Pi_{bd} \preceq \widetilde{\Pi}_{bd}(\eps) \preceq (1+\eps)\Pi_{bd}.
\]
Further, for any $1$-chain $x$, $\widetilde\Pi_{bd}(\eps)\cdot x$ can be computed in
$
\tilde{O}(\beta^3\cdot n\cdot \log n \cdot\log(\frac{n}{\lambda_{\min}(X)\cdot\eps}))
$
time, where $\beta$ is the rank of the first homology group of $K$, $n$ is the total number of simplices in $X$, and $\lambda_{\min}(X)$ is the smallest nonzero eigenvalue of $L_1^{up}(X)$.
\end{restatable}

\subparagraph{} A key technical challenge to achieve our projection operator onto the boundary space is computing a projection into the space of harmonic chains, formalized in part (ii) of the following lemma. Note that our approximation guarantee for projection into the harmonic space is weaker than the one for projection into the boundary space (more on this in the overview).

\begin{restatable}{lemma}{harmonicprojection}
\label{lem:harmonic_projection}
Let $K$ be a subcomplex of a collapsible simplicial complex $X$ with a known collapsing sequence that is linearly embedded in $\R^3$.
Let $\beta$ be the rank of the first homology group of $K$, $n$ be the total number of simplices in $X$, and $\lambda_{\min}(X)$ be the smallest nonzero eigenvalue of $L_1^{up}(X)$.
\begin{enumerate}
    \item [(i)] For any $\eps>0$, there is an
    $
    \tilde{O}(\beta^2 \cdot n \cdot \log n \cdot \log(\frac{n}{\lambda_{min}(X)\cdot \eps}))
    $
    time algorithm to compute an orthonormal set of vectors $\{\T g_1, \ldots, \T g_\beta\}$ such that there exists an orthonormal harmonic basis $\{g_1, \ldots, g_\beta\}$ with $\|g_i - \T g_i\| \leq \eps$ for all $1\leq i\leq\beta$.

    \item [(ii)]
    For any $\eps > 0$, there exists a symmetric matrix $\widetilde{\Pi}_{hr}(\eps)$ such that,
    \[
    \Pi_{hr} -\eps I \preceq \widetilde{\Pi}_{hr}(\eps) \preceq \Pi_{hr} +\eps I,
    \]
    where $\Pi_{hr}$ is the orthogonal projection into the harmonic space.
    Moreover, for any $1$-chain $x$, $\widetilde{\Pi}_{hr}(\eps) \cdot x$ can be computed in
    $
    \tilde{O}(\beta^2 \cdot n \cdot \log n \cdot \log(\frac{n}{\lambda_{min}(X)\cdot \eps}))
    $
    time.
\end{enumerate}

\end{restatable}

Our projection operators into the harmonic and boundary space, along with the projection operator of Cohen et al.~\cite{OneLaplaciansCohen14} into the coboundary space,
give all the projection operators needed to compute the Hodge decomposition of $1$-chains in $K$.

Our harmonic projection operator is built using an orthonormal approximate harmonic basis (part (i) of Lemma~\ref{lem:harmonic_projection}). Dey~\cite{dey2019basis} describes a nearly-linear time algorithm for computing a homology basis for a complex linearly embedded in $\R^3$. Black et al.~\cite{CollapsibleUniverseBlack22} describe a nearly-linear time algorithm for computing a cohomology basis for subcomplexes of collapsible complexes embedded in $\R^3$.  Our harmonic basis, though approximate, can be viewed as a complement to these two results.

\subsection{Paper organization}

In addition to this introduction, the main body of this paper is a background and overview section. To simplify the presentation, the bulk of the technical details are left for the appendix, and the overview provides a high-level description of our approaches as well as the technical challenges and contribution of this paper. In the overview, we included references to the technical lemmas to enable easy access to the technical portion of the paper.

The background section introduces standard definitions of the concepts used in this paper.
We hope this section provides easy lookup for the reader while reading the overview section as well as the technical part of the paper.

\section{Background}
\label{sec:background}

In this section, we review basic definitions from linear algebra and algebraic and combinatorial topology that are used in this paper;  see references~\cite{CampbellBook79, Hatcher, JohnsonHornBook90, StillWell93Book} for further background.

\subsection{Linear Algebra}

\subparagraph{Span, Basis.}
Let $V = \{v_1, \ldots, v_k\}$ be a set of vectors in $\R^n$.  The \EMPH{span} of $V$, denoted $\spn(V)$, is the subspace of $\R^{n}$ of all linear combinations of $V$.
In particular, $V$ spans $\R^n$ if any vector in $\R^n$ is a linear combination of the vectors in $V$.
We say that $V$ is a \EMPH{basis} for its span if the dimension of its span equals the cardinality of $V$.

\subparagraph{Linear map, projection, inverse.} Let $A:\R^{n}\to\R^{m}$ be a linear map, represented by an $m\times n$ matrix. Typically, we don't make a distinction between a linear map and its matrix representation and denote both as $A$. The \EMPH{kernel} of $A$ is $\ker(A):= \{x\in\R^{n} : Ax=0\}$, and the \EMPH{image} of $A$ is $\im(A) = \{Ax : x\in\R^{m}\}$. The \EMPH{rank} of a linear map is the dimension of its image.
\par
We say that $U$ and $V$ \EMPH{orthogonally decompose} $W$, denoted $W = U \oplus V$, if (i) any vector in $U$ is orthogonal to any vector in $V$, and (ii) any vector in $x\in W$ is a unique sum of vectors in $x_{U}\in U$ and $x_{V}\in V$, i.e. $x = x_{U} + x_{V}$. The \EMPH{fundamental theorem of linear algebra} states that $\R^{n} = \im(A^T) \oplus \ker(A)$ and $\R^{m} = \im(A)\oplus\ker(A^T)$, where $A^T$ is the transpose of $A$ obtained by flipping $A$ over its diagonal.  In particular, if $A:\R^{n}\to\R^{n}$ is symmetric (i.e., $A=A^{T}$), then $\R^{n} = \im(A)\oplus\ker(A)$.
\par
A linear map $A:\R^n\rightarrow \R^n$ is a \EMPH{projection} if it is the identity for the vectors in its image, or equivalently, $AA = A$.
The map $A$ is an \EMPH{orthogonal projection} if it maps each point of $\R^n$ to its closest point in $\im(A)$, or equivalently, $A^T = A = AA$.
Note for any subspace $U$ of $\R^n$ there is a unique orthogonal projection into $U$, denoted $\Pi_U$. If $\{u_1,\ldots,u_k\}$ is an orthonormal basis for $U$, the orthogonal projection into $U$ is the linear map $\Pi_{U} = \sum_{i=1}^{k} u_iu_i^{T}$.
\par
If a linear map $A:\R^n\rightarrow\R^m$ is bijective, it has a well-defined \EMPH{inverse} denoted $A^{-1}:\R^{m}\rightarrow\R^{n}$ where $Ax=b\Longleftrightarrow A^{-1}b = x$.  More generally,
the \EMPH{pseudoinverse} of $A$ is the unique linear map $A^{+}:\R^{m}\to\R^{n}$ with the following properties: (i) $AA^{+}A = A$, (ii) $A^{+}AA^{+} = A^{+}$, (iii) $(AA^{+})^T = AA^{+}$, and (iv) $(A^{+}A)^{T} = A^{+}A$. Admittedly, the definition of the pseudoinverse is not very intuitive. A more intuitive description is that $A^{+}$ is the unique linear map with the following properties: (1) $A^{+}$ maps any vector $y\in\im(A)$ to the unique vector $x\in\im(A^{T})$ such that $Ax=y$, and (2) $A^{+}$ maps any vector $y\in\ker(A^{T})$ to 0. While it is not true in general that $(A+B)^{+} = B^{+}+A^{+}$ for linear maps $A$ and B, this is true if $A^{T}B=B^{T}A=0$; see Campbell \cite{CampbellBook79}, Theorem 3.1.1.

\subparagraph{Matrix norm, singular values, Loewner order.}
A symmetric matrix $A$ is \EMPH{positive semidefinite} if $x^TAx\geq 0$ for each $x\in\R^{n}$. The \EMPH{Loewner Order} is a partial order on the set of $n\times n$ symmetric matrices. For symmetric matrices $A$ and $B$, we say $A\preceq B$ if $B-A$ is positive semidefinite.
\par
Let $x\in\R^{n}$. Let $p$ be a positive integer. The \EMPH{p-norm} of $x$ is $\|x\|_p = \left(\sum_{i=1}^{n}|x[i]|^{p}\right)^{\frac{1}{p}}$.
We use the $1$-norm and $2$-norm in this paper. An important fact we will use throughout this paper is that $\|x\|_2\leq \|x\|_1 \leq \sqrt{n} \|x\|_2$. For any norm $\|\cdot\|$ on $\R^{n}$, there is an accompanying \EMPH{operator norm} of a matrix $A$ defined $\| A \| = {\max}_{x:\|x\|=1} \| Ax\|$, or equivalently, $\|A\| = \underset{x\neq 0}{\max}\left({\| Ax \|}/{\|x\|}\right)$. Unless otherwise specified, all norms in this paper will be the 2-norm.
\par
The \EMPH{singular value decomposition} of $A:\R^{n}\to\R^{m}$ for $m\geq n$ (resp. $m\leq n$) is a set of $n$ (resp. $m$) orthornomal vectors $\{u_1,\ldots,u_n\}\subset\R^{m}$ called \EMPH{left singular vectors}, $n$ (resp. $m$) orthornomal vectors $\{v_1,\ldots,v_n\}\subset\R^{n}$ called \EMPH{right singular vectors}, and $n$ (resp $m$) real numbers $\{\sigma_1,\ldots,\sigma_n\}\subset\R$ called \EMPH{singular values} such that $A = \sum_{i=1}^{n} \sigma_iu_iv_i^T$. The \EMPH{condition number} of a linear map $A:\R^{n}\to\R^{n}$ is $\kappa(A) = |\sigma_{\max}(A)|/|\sigma_{\min}(A)|$, where $\sigma_{\max}(A)$ and $\sigma_{\min}(A)$ are the largest and smallest non-zero singular values of $A$.
\par
The \EMPH{eigenvectors} and \EMPH{eigenvalues} of a matrix $A:\R^{n}\to\R^{n}$ are $n$ vectors $\{v_1,\ldots,v_n\}\subset\R^{n}$ and $n$ real numbers $\{\lambda_1,\ldots,\lambda_n\}$ such that $Av_i = \lambda_i v_i$. The singular values and right singular vectors (resp. left singular values) of a matrix $A:\R^{n}\to\R^{n}$ are the square roots of the eigenvalues and eigenvectors of $A^TA$ (resp. $AA^{T}$). If a matrix $A$ is symmetric, the eigenvectors of $A$ are orthogonal, and the eigenvectors and eigenvalues of $A$ are the left and right singular vectors and the singular values.

\subparagraph{Determinant, Cramer's rule, unimodularity.}
For any $1\leq i\leq n$, the \EMPH{determinant} of an $n\times n$ matrix $A = [a_{i,j}]_{1\leq i, j\leq n}$ can be defined via its \EMPH{Laplace expansion} as
$
\det(A) = \sum_{j=1}^{n}{\left((-1)^{i+j}\cdot a_{i, j}\cdot \det(A_{i, j})\right)},
$
where $A_{i,j}$ is the $(n-1)\times(n-1)$ matrix obtained by removing the $i$th row and $j$th column of $A$.  It is well known that $\det(A) \neq 0$ if and only if $A$ is bijective.  In that case, \EMPH{Cramer's rule} give an explicit formula for the solution of the linear system $Ax = b$, which is $x[i] = \det(A_i)/\det(A)$ where $A_i$ is the matrix obtained by replacing the $i$\textsuperscript{th} column of $A$ with $b$.

An $n\times n$ matrix $A$ is \EMPH{unimodular} if $\det(A) \in \{-1, +1\}$.  By Cramer's rule, $Ax = b$ has an integer solution if $A$ is unimodular and $A$ and $b$ have integer coefficients.  An $n\times m$ matrix $B$ is \EMPH{totally unimodular} if for any square submatrix $A$ of $B$, $\det(A)\in \{-1, 0, +1\}$. The 1-boundary matrix of a simplicial complex (defined below) is totally unimodular~\cite{Shrijver1986LinIntProgBook}.

\subsection{Topology}

\subparagraph{Simplicial complexes.} A \EMPH{simplicial complex} $K$ is a set of finite sets such that if $\tau\in K$ and $\sigma\subset\tau$, then $\sigma\in K$. A \EMPH{subcomplex} of $K$ is a subset $L\subset K$ such that $L$ is a simplicial complex. The \EMPH{vertices} of $K$ is the set $\cup_{\sigma\in K}\sigma$. We assume there is a fixed but arbitrary order $(v_1,\ldots,v_n)$ on the vertices of $K$.
\par
An element $\sigma\in K$ with $|\sigma|=d+1$ is a \EMPH{d-simplex}. A 0-simplex is a \EMPH{vertex}, a 1-simplex is an \EMPH{edge}, a 2-simplex is a \EMPH{triangle}, and a 3-simplex is a \EMPH{tetrahedron}. The set of all $d$-simplices in $K$ is denoted $K_d$.  For two simplices $\tau\subset\sigma$, we say that $\tau$ is a \EMPH{face} of $\sigma$.

\subparagraph{Hodge decomposition, homology, cohomology.}
The \EMPH{d\textsuperscript{th} chain group} of a simplicial complex $K$ is the vector space $C_d(K)$ over $\R$ with orthonormal basis $K_d$, and an element of $C_d(K)$ is a \EMPH{d-chain}. The \EMPH{d\textsuperscript{th} boundary map} is the linear map $\boundary_d:C_d(K)\to C_{d-1}(K)$ defined $\boundary_d\sigma = \sum_{i=0}^{d}(-1)^{i}(\sigma\setminus\{v_{k_i}\})$ for each simplex $\sigma = \{v_{k_0},\ldots,v_{k_d}\}\in K_d$, where we assume $v_{k_i}<v_{k_j}$ for $i<j$. The \EMPH{d\textsuperscript{th} coboundary map} is $\boundary_{d+1}^T:C_d(K)\to C_{d+1}(K)$. Elements of $\ker\boundary_d$ (resp. $\ker\boundary_{d+1}^T$) are \EMPH{cycles} (resp. \EMPH{cocyles}), and elements of $\im\boundary_{d+1}$ (resp. $\im\boundary_{d}^{T}$) are \EMPH{boundaries} or \EMPH{null-homologous cycles} (resp. \EMPH{coboundaries}.) Two cycles (resp. cocycles) $\gamma_1$ and $\gamma_2$ are \EMPH{homologous} (resp. \EMPH{cohomologous}) if their difference $\gamma_1-\gamma_2$ is a boundary (resp. coboundary.)
\par
The \EMPH{d\textsuperscript{th} Laplacian} is the linear map $L_d : C_d(K)\to C_d(K)$ defined $L_d = \boundary_{d}^{T}\boundary_d + \boundary_{d+1}\boundary_{d+1}^{T}$. The \EMPH{d\textsuperscript{th} up-Laplacian} is the linear map $L_d^{up} = \boundary_{d+1}\boundary_{d+1}^{T}$, and the \EMPH{d\textsuperscript{th} down-Laplacian} is the linear map $L_d^{down} = \boundary_{d}^{T}\boundary_{d}$.
\par
A key fact of algebraic topology is that $\boundary_d\boundary_{d+1} = 0$, hence $\im\boundary_{d+1}\subset\ker\boundary_d$, and $\im\boundary^T_d \subset \ker\boundary^T_{d+1}$.
The \EMPH{d\textsuperscript{th} homology group} is the quotient group $H_d(K) = \ker\boundary_d/\im\boundary_{d+1}$, and the \EMPH{d\textsuperscript{th} cohomology group} is the quotient group $H^{d}(K) = \ker\boundary^T_{d+1}/\im\boundary^T_d$.  Since $\im\boundary_d^T\oplus\ker\boundary_d$ and $\im\boundary_{d+1}\oplus\ker\boundary_{d+1}^{T}$ are two orthogonal decompositions of the $d$-chain space, the d\textsuperscript{th} homology group and the d\textsuperscript{th} cohomology group have the same rank, which is the \EMPH{d\textsuperscript{th} Betti number} of the complex, denoted $\beta_d(K)$.  We say two cycles are \EMPH{homologous} (resp.~\EMPH{cohomologous}) if they are in the same homology (resp.~cohomology) class, or equivalently, if their difference is a boundary (resp.~coboundary.)
\par
The \EMPH{Hodge Decomposition} is the orthogonal decomposition of the $d$\textsuperscript{th} chain group into $C_d(K) = \im(\boundary_{d+1}) \oplus \ker(L_d) \oplus \im(\boundary^{T}_{d})$. The subspace $\ker(L_d)$ are the \EMPH{harmonic chains}. Thus, any chain $x\in C_d(K)$ can be uniquely written as the sum $x = x_{bd} + x_{hr} + x_{cbd}$ where $x_{bd}\in\im(\boundary_{d+1})$, $x_{hr}\in\ker(L_d)$, and $x_{cbd}\in\im(\boundary^{T}_{d})$.


\par
A \EMPH{d-boundary basis}, \EMPH{d-coboundary basis} and \EMPH{d-harmonic basis} are bases for the boundary, coboundary and harmonic spaces.
A \EMPH{d-homology basis} is a maximal set of cycles such that no linear combination of these cycles is a boundary.
Similarly, a \EMPH{d-cohomology basis} is a maximal set of cocycles such that no linear combination of these cocycles is a coboundary.
We have the following fact.
\begin{fact}
\label{fact:homology_basis_to_harmonic_basis}
A set of cycles (resp.~cocycles) is a homology (resp.~cohomology) basis if and only if their projection into the harmonic space is a harmonic basis.
\end{fact}
Two cycles (resp.~cocycles) are homologous (resp.~cohomologous) if they have the same harmonic part, as then their difference is a boundary (resp.~coboundary). Accordingly, the previous fact implies that for any cycle (resp.~cocycle) $x$ and any homology basis (resp.~cohomology basis) $\Gamma$, there is a unique linear combination of the elements of $\Gamma$ that is homologous (resp.~cohomologous) to $x$; this is the linear combination of $\Gamma$ with the same harmonic component as $x$.
\par
A useful property of cohomology bases is they can be used to tell when two cycles are homologous, as described by the following fact.
\begin{fact}[Busaryev et al.~\cite{busaryev2012}]
\label{fact:homology_annotation}
    Let $x$ and $y$ be cycles (resp.~cocycles), and let $P$ be a cohomology basis (resp.~homology basis.) Then $y$ is homologous (resp.~cohomologous) to $x$ if and only if $x\cdot p = y\cdot p$ for all $p\in P$.
\end{fact}

\subparagraph{Collapsibility.}
Let $K$ be a simplicial complex, $\sigma$ a $d$-simplex of $K$, and $\tau$ a $(d-1)$-simplex of $K$ that is a face of $\sigma$.  If $\tau$ is not the face of any other simplex, we say that $K$ \EMPH{collapses} into $K\backslash\{\sigma, \tau\}$; we refer to $(\sigma,\tau)$ as a \EMPH{collapse pair}.  Moreover, we say that a complex collapses to itself.  Inductively, we say that a complex $K$ \EMPH{collapses} into a complex $K'$ if there is a complex $K''$ such that $K$ collapses to $K''$ and $K''$ collapses to $K'$.  We say that a complex $K$ is \EMPH{collapsible} if it collapses to a single vertex.

When a complex $K$ collapses to a complex $K'$, we obtain a sequence of complexes $K = K_0\supset K_1 \supset \ldots \supset K_t = K'$, where for each $1\leq i\leq t$, $K_i$ can be obtained from $K_{i-1}$ by removing one collapse pair.  We refer to this sequence as a \EMPH{collapsing sequence}.
The complexes $K$ and $K'$ are homotopy equivalent if one collapses to the other, thus, $K$ and $K'$ have isomorphic homology group. In particular, a collapsible complex has trivial homology groups in every nonzero dimension.

\subparagraph{Embeddability.}
A $d$-dimensional simplicial complex $K$ is \EMPH{embedded} if $K\subset R$ for $R$ a triangulation of $\R^{d+1}$. Furthermore, $K$ is \EMPH{linearly embedded} if there is a homeomorphism from the underlying space $|R|$ to $\R^{d+1}$ that is linear on each simplex, i.e. each 1-simplex is mapped to a line segment, each 2-simplex is mapped to a triangle, etc. All embedded complexes in this paper will be linearly embedded.
\par
We will make use of the dual graph of an embedded complex. Informally, the \EMPH{dual graph} of an embedded complex is the graph $K^*$ with vertices that are the connected components of $R\setminus K$ and edges between two vertices if there is a $d$-simplex in $K$ incident to both connected components. Alternatively, the dual graph can be defined with vertices corresponding to a generating set of $d$-cycles of $K$. For this construction, see the definition of \EMPH{Lefschetz set} in the paper \cite{CollapsibleUniverseBlack22}.

\section{Overview}
Let $X$ be a collapsible simplicial complex embedded in $\R^3$, and let $K\subseteq X$ be a subcomplex of $X$.  We study two closely related problems: (i) computing the Hodge decomposition of the $1$-chains of $K$, and (ii) solving a linear system $L_1 x = b$, where $L_1$ is the $1$-Laplacian of $K$ (in the overview section, all the operators are with respect to $K$ unless mentioned otherwise.) These two problems are related, as our approximate Laplacian solver uses an approximate Hodge decomposition of the input vector $x$. More generally, understanding the Hodge decomposition is key to understanding this paper as many proofs rely on some property of the Hodge decomposition. Therefore, we begin our overview with an introduction to the Hodge decomposition.

\subsection{The Hodge Decomposition} The Hodge decomposition is a decomposition of the chain group $C_d(K)$ in terms of the kernels and images of the boundary operators $\boundary_d$ and $\boundary_{d+1}$ and their transposes. Specifically, the problems in this paper consider the first chain group $C_1(K)$, the two boundary operators $\partial_2$ and $\partial_1$, and their corresponding coboundary operators $\partial_2^T$ and $\partial_1^T$.  The boundary operator $\partial_2$ maps each (oriented) triangle to the edges in its boundary; similarly, $\partial_1$ maps each edge to its two endpoints.  A key fact is that $\partial_1\partial_2 = 0$, or equivalently, $\im(\partial_2) \subseteq \ker(\partial_1)$. This implies $\im(\partial_2)$ is orthogonal to $\im(\partial^T_1)$.  The former subspace $\im(\partial_2)$ is called the \EMPH{boundary subspace}, and the latter subspace $\im(\partial^T_1)$ is called the \EMPH{coboundary subspace}.  If $K$ has trivial $1$-homology, then $\im(\boundary_2) = \ker(\boundary_1)$, and the boundary and coboundary spaces give a full orthogonal decomposition of $C_1(K)$, called the \EMPH{Helmholtz decomposition}.  Otherwise, there is a third subspace orthogonal to both the boundary and coboundary subspaces, called the \EMPH{harmonic subspace}. The harmonic subspace is exactly $\ker(L_1) = \ker(\partial_1)\cap\ker(\partial_2^T)$.
\par
The boundary, coboundary, and harmonic subspaces give a full orthogonal decomposition of $C_1(K)$ called the \EMPH{Hodge decomposition}, which generalizes the Helmholtz decomposition.  Thus, we can express any $1$-chain $x$ as $x = x_{cbd} + x_{bd} + x_{hr}$, where $x_{cbd}$, $x_{bd}$ and $x_{hr}$ are the coboundary, boundary and harmonic part of $x$ and are pairwise orthogonal. The chains $x_{bd} + x_{hr}$ and $x_{cbd} + x_{hr}$ are called the \EMPH{cyclic} and \EMPH{cocyclic} parts of $x$ respectively.  Similarly, the space spanned by harmonic and boundary chains is called the \EMPH{cycle space}, and the space spanned by harmonic and coboundary chains is called the \EMPH{cocycle space}. It is implied by $\partial_1\partial_2 = 0$ that the cycle space and cocycle space are the kernels of $\partial_1$ and $\partial_2^T$, respectively. The following figure is an illustration of the Hodge decomposition.  Boundary, coboundary, harmonic, cycle, and cocycle spaces are shown using the abbreviations bd, cbd, hr, cyc, and cocyc respectively.

\begin{figure}[H]
    \centering
    \includegraphics[height=1.4in]{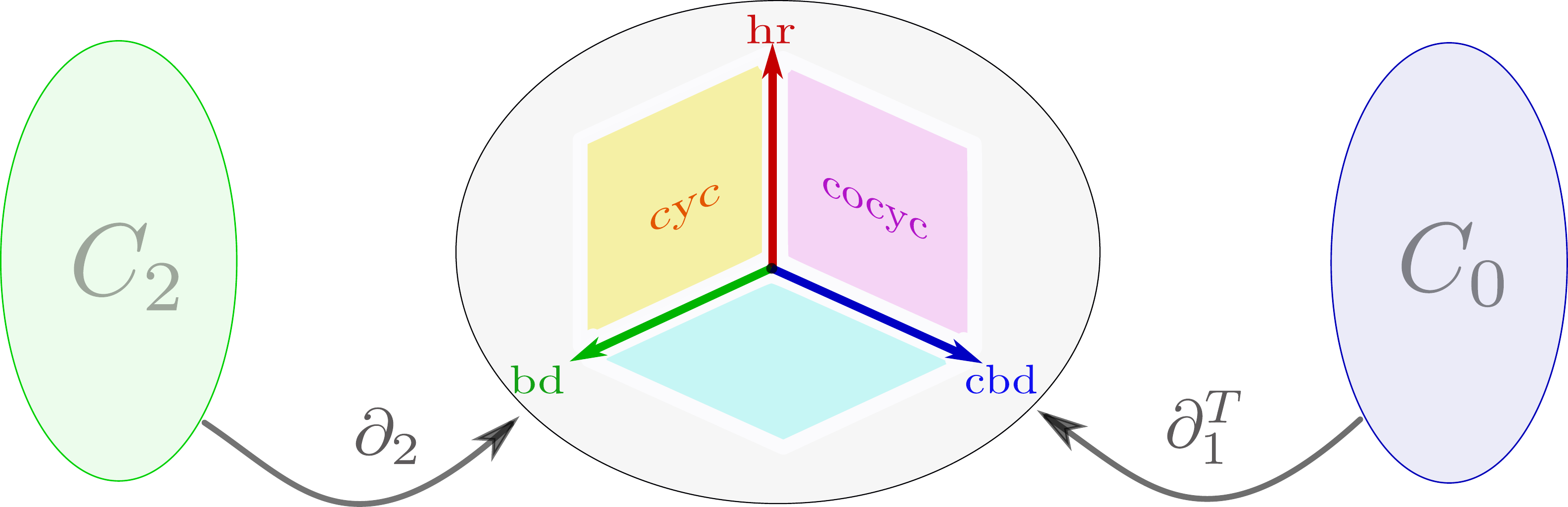}
    \label{fig:hodge}
\end{figure}

To compute the Hodge decomposition, one seeks orthogonal projection operators into the coboundary, boundary and harmonic subspaces.  Let $\Pi_{cbd}$, $\Pi_{bd}$, and $\Pi_{hr}$ denote these projection operators. Cohen et al.~show that for any $1$-chain $x$, its projection into the coboundary space, $\Pi_{cbd} x$, and cycle space, $\Pi_{cyc} x$, can be approximated quickly with operators $\widetilde\Pi_{cbd}$ and $\widetilde\Pi_{cyc}$. These projection operators are a key ingredient of their $1$-Laplacian solver, as well as the more recent $1$-Laplacian solver described by Black et al.; however, both papers are restricted to cases where the first homology group $H_1(K) = 0$. In this paper, we show that for any $x$, its projection into the boundary space, $\Pi_{bd} x$, can also be approximated quickly. This new projection operator will allow us to generalize the $1$-Laplacian solver of Black et al.~to complexes with arbitrary first homology. We also give an approximate projection operator into the harmonic space, but our approximation guarantee for this projection operator is weaker (more below).

\subsection{Laplacian Solvers.}

The $1$-Laplacian matrix is defined $L_1 = \partial_2\partial_2^T + \partial_1^T\partial_1$. To solve a linear system $L_1 x = b$, one seeks to approximate $L_1^+$, the pseudoinverse of $L_1$. As the images of $\partial_2\partial_2^T$ and $\partial_1^T\partial_1$ are orthogonal, then $L_1^{+} = (\partial_2\partial_2^T)^{+} + (\partial_1^T\partial_1)^+$ (see Campbell~\cite[Theorem 3.1.1]{CampbellBook79}). Therefore,  one can approximate $L_1^{+}$ by approximating $(\partial_2\partial_2^T)^+$ and $(\partial_1^T\partial_1)^+$ individually.
Computing $(\partial_1^T\partial_1)^+$ is purely a graph problem as $\partial_1$ is only defined with respect to the vertices and edges of a complex. Cohen et al.~show how to approximate  $(\partial_1^T\partial_1)^+$ for general complexes \cite[Lemma 3.2]{OneLaplaciansCohen14}. Approximating $(\partial_2\partial_2^T)^+$ is a more challenging problem that requires taking into account the relationship between triangles and the edges. Our algorithm for approximating $(\partial_2\partial_2^T)^+$ relies on our new boundary projection operator, the collapsibility of $X$, and the embedding of $X$ in $\R^3$.
\par
Cohen et al.~show how to approximate $(\partial_2\partial_2^T)^+$ for collapsible complexes embedded in $\R^{3}$. Black et al.~generalize their work to obtain an approximate solver for a subcomplex of a collapsible complex in $\R^{3}$ provided the subcomplex has trivial homology.  Their solver is based on the following general lemma regarding approximations of $(BB^T)^+$ for a general matrix $B$.

\begin{restatable*}[Black et al.~\cite{CollapsibleUniverseBlack22}, Lemma 4.1]{lemma}{approximationbbt}
\label{lem:black_etal_bbt}
    Let $B$ be a linear operator, let $0\leq \varepsilon < 1$,
    and let $\T\Pi_{\im(B)}$ and $\T\Pi_{\ker^{\perp}(B)}$ be symmetric matrices such that
    $(1-\varepsilon) \Pi_{\im(B)} \preceq \T\Pi_{\im(B)}\preceq \Pi_{\im(B)}$, and
    $(1-\varepsilon)\Pi_{\ker^{\perp}(B)}\preceq \T\Pi_{\ker^{\perp}(B)}\preceq \Pi_{\ker^{\perp}(B)}$.
    Also, let $U$ be a linear map such that for any $y\in \im(B)$, $BUy = y$.  We have
    \[
    (1-(2\kappa+1)\varepsilon)(BB^T)^+ \preceq
    \T\Pi_{\im(B)}U^T \T\Pi_{\ker^{\perp}(B)} U \T\Pi_{\im(B)} \preceq
    (1+\kappa\varepsilon)(BB^T)^+,
    \]
    where $\kappa$ is the condition number of $BB^T $ within the image of $B$.
\end{restatable*}

This lemma shows the following linear operators are sufficient for approximating $(\partial_2\partial^T_2)^+$.

\begin{enumerate}
    \item [(i)] An operator $U$ that for $1$-boundaries $y\in\im(\partial_2)$ returns a $2$-chain $x = Uy$ such that $\boundary_2 x = y$. For other vectors $z\notin\im(\partial_2)$, $U$ can return anything as long as $U$ is still linear.
    \item [(ii)] An approximate orthogonal projection operator into $\im(\partial_2^T)$, the coboundary space of $2$-chains.
    \item [(iii)] An approximate orthogonal projection operator into $\im(\partial_2)$, the boundary space of $1$-chains.
\end{enumerate}

Black et al.~describe an algorithm for computing $U$ that uses the collapsibility and embedding of the supercomplex $X$.  Cohen et al.~show that the 2-coboundary space of embedded complexes is dual to the 1-cycle space of the dual graph, hence projection into this space can be approximated using $\T\Pi_{cyc}$. Finally, lacking an approximate projection into the boundary space of 1-chains, they needed to assume that their complex has trivial first homology (i.e. that $\im(\boundary_2) = \ker(\boundary_1)$) so that they can instead use the projection operator into the cycle space of Cohen et al. The boundary projection operator described in this paper allow us to remove that assumption to obtain a solver for any subcomplex $K$ of $X$.  The running time of our new solver polynomially depends on the rank of the homology group and nearly-linearly depends on the size of the complex. We give a complete analysis of our solver in Appendix~\ref{sec:solver}.

\subparagraph{} In the rest of this section, we sketch the high level ideas for computing our approximate projection operators.  But before we can do that, we need to explain the two notions of approximations that are used in this paper.

\subsection{Loewner order approximation}
We use the Loewner order on positive semidefinite matrices to specify the approximation quality of our projection and pseudoinverse operators.
We see two types of approximation guarantees in this paper for an operator $A$: \textit{\textbf{input-relative error bounds}} of the form $-\eps I \preceq A - \T A \preceq \eps I$ and \textit{\textbf{output-relative error bounds}} of the form $-\eps A \preceq A - \T A \preceq \eps A$.
Note for any vector $x$, an input relative error bound implies
$\|(A - \T A)x\| \leq \eps \|x\|$---the error is bounded relative to the size of the input $x$---while an output-relative error bound implies $\|(A-\T A)x\|\leq \eps \|A x\|$---the error is bounded relative to the size of the output $Ax$. An approximate operator with a small input-relative error can have arbitrarily large output-relative error, for example when $x$ is in the kernel of $A$. Further, output-relative error bounds are stronger if the norm of $\|A\|$ is at most one, i.e.~$\|Ax\|\leq \|x\|$, which is the case for the orthogonal projection operators of the Hodge decomposition.

We achieve an output-relative error bound for our approximation of $(L_1[K])^+$. Further, we achieve an output-relative error bound for our approximation $\T\Pi_{bd}$ of $\Pi_{bd}$, but an input-error bound for our approximation $\T\Pi_{hr}$ of $\Pi_{hr}$:
\begin{equation}
\label{eqn:bd_proj}
-\eps I\preceq -\eps\Pi_{bd}\preceq \Pi_{bd} - \T\Pi_{bd}(\eps)  \preceq
\eps\Pi_{bd} \preceq \eps I,
\end{equation}
and
\begin{equation}
\label{eqn:hr_proj}
-\eps I \preceq \Pi_{hr} - \T\Pi_{hr}(\eps) \preceq \eps I.
\end{equation}
Previously, Cohen et al.~had shown approximation operators $\T\Pi_{cbd}$ and $\T\Pi_{cyc}$ for projecting into the coboundary and cycle spaces with output-relative error bounds:
\begin{equation}
\label{eqn:cbd_proj}
-\eps I \preceq -\eps\Pi_{cbd}\preceq \Pi_{cbd} - \T\Pi_{cbd}(\eps) \preceq \eps\Pi_{cbd} \preceq \eps I,
\end{equation}
and
\begin{equation}
\label{eqn:cyc_proj}
-\eps I \preceq -\eps\Pi_{cyc}\preceq \Pi_{cyc} - \T\Pi_{cyc}(\eps) \preceq \eps\Pi_{cyc} \preceq \eps I.
\end{equation}
We use these operators multiple times in our algorithms.  For simplification, we drop the explicit mention of the parameter $\eps$ when it is clear from the context in the overview.

\subsection{Projection operators.}

We first describe our algorithm for computing $\T\Pi_{hr}$ (an overview of Section~\ref{sec:harm_of_cohom} and Section~\ref{sec:harmonic_proj}).  Based on that and the operator $\T\Pi_{cbd}$ of Equation (\ref{eqn:cbd_proj}), we show how to compute $\T\Pi_{bd}$ (an overview of Section~\ref{sec:boundary_proj}).

\subparagraph{Harmonic projection.}
We compute our approximate harmonic projection operator $\T \Pi_{hr}$ by computing an approximate orthonormal basis $\T G = \{\T g_1, \ldots, \T g_\beta\}$ of the harmonic space.
We then define the approximate projection into the harmonic space to be the linear map $\T \Pi_{hr} = \sum_{i=1}^{\beta} \tilde{g_i}\tilde{g_i}^{T}$.

To compute $\T G$, our algorithm starts with a cohomology basis $P=\{p_1, \ldots, p_\beta\}$; the algorithm for computing $P$ is given at the end of this section.  From $P$, it computes $\T H = \{\T h_1, \ldots, \T h_\beta\}$, where $\T h_i = \T \Pi_{cyc} p_i$ and $\T\Pi_{cyc}$ is the projection operator of Equation (\ref{eqn:cyc_proj}). The set $\T H$ is an approximate harmonic basis, but it is not orthonormal. Next, we normalize $\T H$ to obtain $\T N = \{\T h_1/\|\T h_1\|, \ldots, \T h_\beta/\|\T h_\beta\|\}$. Finally, we run Gram-Schmidt on $\T N$ to obtain $\T G$.

To see why $\T{G}$ is an approximate basis for the harmonic space, let us consider a much easier analysis assuming we can use the exact projection in the cycle space $\Pi_{cyc}$ instead of the approximate projection $\T \Pi_{cyc}$. Instead of $\T H$, $\T N$ and $\T G$, let $H = \{h_1, \ldots, h_\beta\}$, $N = \{h_1/\|h_1\|, \ldots, h_\beta/\|h_\beta\|\}$ and $G=\{g_1, \ldots, g_\beta\}$ be the sets of vectors we obtain when we use the exact projection operator.  In that case,
$h_i = \Pi_{cyc} p_i$ is the harmonic part of $p_i$; this is because $p_i$ is a cocycle, so projecting it into the cycle space is the same as projecting it into the harmonic space.  It follows from Fact \ref{fact:homology_basis_to_harmonic_basis} in Section \ref{sec:background} that $G$ is an exact orthonormal basis of the harmonic space, thus it defines an exact projection operator into the harmonic space.

In the real scenario where we work with the approximate projection operator $\T\Pi_{cyc}$, two undesirable things can happen. First, we can no longer guarantee that the vectors in $\T N$ are purely harmonic, as the error introduced by the approximate operator $\T\Pi_{cyc}$ may be part boundary. However, this is not an issue, as we can make the boundary components of $\T N$ sufficiently small by approximating $\T\Pi_{cyc}$ more accurately. Second, and more importantly for our application, the spaces spanned by $N$ and $\T N$ can be very different, even if the vectors $N$ and $\T N$ are pairwise close. As an example, imagine that we have two pairs of vectors $N=\{\eta_1,\eta_2\}$ and $\T N = \{\T \eta_1,\T \eta_2\}$ such that $\|\eta_i - \T \eta_i\| < \eps$ for $i=1,2$. We might guess that the two spaces spanned by $N$ and $\T N$ are similar as the vector are close, but if $\eta_1$ and $\eta_2$ are also close, then the two vectors spaces can be drastically different. Figure \ref{fig:delta_independence} gives an illustration of this, where $N$ is the set of blue vectors and $\T N$ is the set of red vectors.  As illustrated in the figure, the space spanned by $N$ and the space spanned by $\T N$ can be drastically different.

\begin{figure}[h!]
    \centering
    \includegraphics[height=2in]{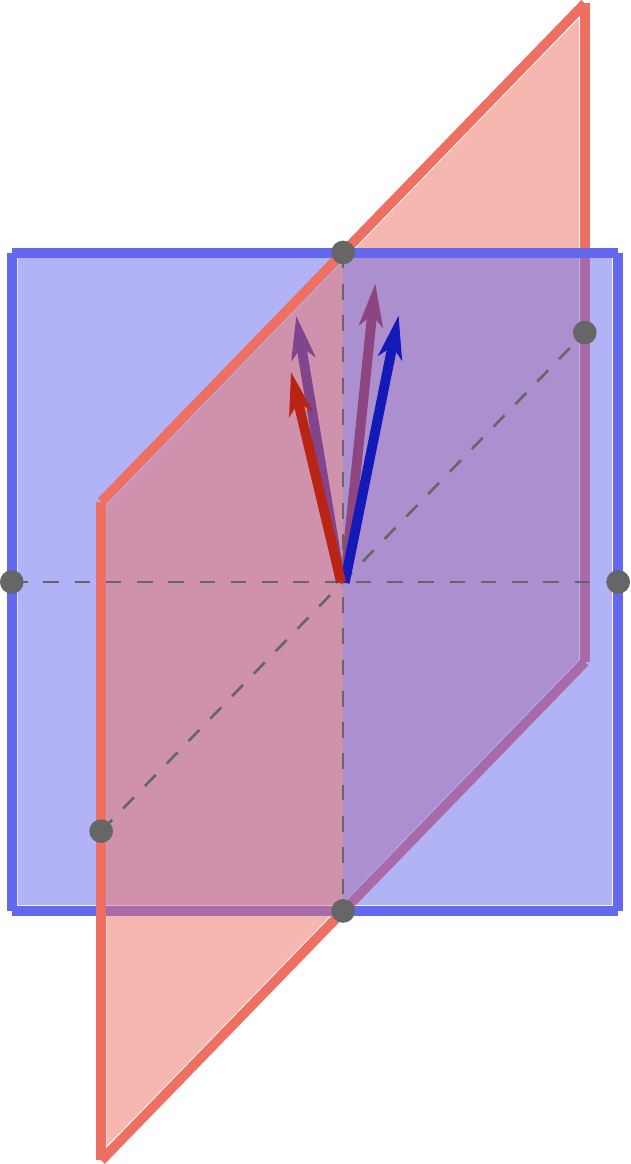}
    \caption{Pairwise closeness between a set of vectors $N$ and $\T N$ is not enough to guarantee the spaces spanned by $N$ and $\T N$ are close! The red and blue vectors are pairwise close, but the spaces they span are very different.}
    \label{fig:delta_independence}
\end{figure}

 We can remedy this if we approximate $\T N$ within a sufficiently small error $\eps$ of $N$, but this new error bound needs to take into account the similarity of the vectors in $N$. The question is how accurately we need to approximate $\T\Pi_{cyc}$ to obtain a sufficently small approximation error for $\T N$. To answer this question, we define a measure of linear independence of $N$ called its $\delta$-independence.
Formally, we say that $N$ is \EMPH{$\boldsymbol{\delta}$-independent} if each vector $h_i/\|h_i\|\in N$ is at distance at least $\delta$ from the span of the other vectors of $N$. Intuitively, larger $\delta$ means $N$ is more independent, in the sense that the elements are well-separated. The smaller the $\delta$, the more accurately we need to approximate $\T\Pi_{cyc}$ to ensure that $N$ and $\T N$ will span similar spaces. This intuition is summarized by the following lemma, showing the error in projection into $N$ as a function of $\delta$, $\eps$, and $\beta$, where $\eps$ bounds the difference between $N$ and $\T N$.

\begin{restatable*}{corollary}{approximategramschmidt}
\label{cor:approximate_gram_schmidt_operator}
    Let $0< \delta < 1$, and let $0<\eps<\left(\frac{\delta}{8\beta}\right)^{\beta}$. Let $N=\{\eta_1,\ldots,\eta_\beta\}$ be a set of $\delta$-linearly independent unit vectors, and let $\T N=\{\T \eta_1,\ldots,\T \eta_\beta\}$ be a set of unit vectors such that $\| \eta_i - \T \eta_i \| < \eps$.
    Let $G=\{g_1,\ldots,g_\beta\}$ be the orthonormal basis that is the output of running Gram-Schmidt on $N$, and let $\T G = \{\T g_1,\ldots,\T g_\beta\}$ be the output of running Gram-Schmidt on $\T N$.
    Then $\left\| \Pi_{\spn N\vphantom{\Tilde{N}}} - \Pi_{\spn\Tilde{N}} \right\| = \left\| \sum_{i=1}^{\beta} g_ig_i^T - \T g_i\T g_i^T \right\| < 2\cdot \beta\cdot \left(\frac{8\beta}{\delta}\right)^\beta\eps$.
\end{restatable*}

The difficulty here is actually determining a lower bound on the $\delta$-independence of $N$. We have access to the cohomology basis $P$, but we need the (normalized) harmonic parts of $P$ to be $\delta$-independent. Note that $P$ can be composed of vectors that are very strongly independent, yet their harmonic parts may only be weakly independent, for example, when the vectors of $P$ have similar harmonic parts but very different coboundary parts.
\par
We show that if $P$ is composed of integer vectors with maximum length $p_{\max}$, then $P$ being linearly independent implies that $H$ is $\delta$-independent for a $\delta \sim 1/(p_{\max}\cdot n_1)^{\beta}$, where $n_1$ is the number of edges in $K$.
In addition to the properties of $P$, our proof of Lemma~\ref{lem:delta_indep_harm} relies on the total unimodularity of $\partial_1$.

\begin{restatable*}{lemma}{deltaindependentharmonic}
\label{lem:delta_indep_harm}
Let $K$ be a simplicial complex with $n_1$ edges such that $H_1(K) = \beta$.
Let $\{p_1, \ldots, p_\beta\}$ be a $1$-cohomology basis for $K$ such that each $p_i$ is an integer vector with maximum Euclidean norm $p_{\max}$.
Let $h_i$ be the harmonic part of $p_i$ for $1\leq i \leq \beta$.
Then
\begin{enumerate}
    \item [(i)] $\|h_i\| \geq 1/(\sqrt{n_1}\cdot p_{\max})^{\beta}$ for each $1\leq i \leq \beta$, and
    \item [(ii)] $\{h_1/\|h_1\|, \ldots, h_\beta/\|h_\beta\|\}$ is $\left(1/(\sqrt{n_1}\cdot p_{\max})^{\beta}\right)$-independent.
\end{enumerate}

\end{restatable*}

The question remains of how to find $P$. Dey~\cite{dey2019basis} describes a nearly-linear time algorithm for computing a homology basis composed of vectors with coordinates in $\{-1, 0, +1\}$.  Black et al.~\cite{CollapsibleUniverseBlack22} describe an operator $C$ that when applied to a homology basis returns a cohomology basis. We use the cohomology basis $P$ obtained by applying the operator $C$ to Dey's homology basis. The proof of Corollary~\ref{cor:cohom_harmonic_prop} shows that $P$ is composed of vectors whose lengths $p_{\max}$ are bounded above by a polynomial function of the number of simplices of $X$ and $1/\lambda_{\min}(X)$; we prove this lemma by combining a bound on the length of the homology basis with a bound on the operator norm $\|C\|$ (Lemma~\ref{lem:bound_Cgamma}). Passing this cohomology basis $P$ to the algorithm above, we obtain $\T H$, $\T N$, $\T G$, $\T\Pi_{hr}$ as desired. The exact approximation quality of the approximate harmonic basis and approximate harmonic projection are given in Lemma \ref{lem:harmonic_projection} in the introduction.

\subparagraph{Boundary projection.}
It follows from the Hodge decomposition that the the projection into the boundary space can be written $\Pi_{bd} = I - \Pi_{cbd} - \Pi_{hr}$. We have approximate projections $\T\Pi_{cbd}$ and $\T\Pi_{hr}$ with input-relative error bounds (Equations (\ref{eqn:cbd_proj}) and (\ref{eqn:hr_proj}) respectively),  so we immediately obtain a boundary projection with input-relative error bound
\[
\overline\Pi_{bd} = I - \T\Pi_{cbd} - \T\Pi_{hr} \Longrightarrow -\eps I \preceq \Pi_{bd} - \overline\Pi_{bd} \preceq \eps I.
\]
However, we need a boundary projection operator with an output-relative bound for our solver.
Unfortunately, the operator $\overline\Pi_{bd}$ can have arbitrarily bad output-relative error.
Specifically, for any vector $x$ that is orthogonal to the boundary space, this operator has unbounded output-relative error as $\Pi_{bd} x = 0$.

We instead use $\overline\Pi_{bd}$ as a starting point for a projection operator with bounded output-relative error
.  To that end, let's revisit the issue of input vs.~output relative error.  Let $x = x_{bd} + x_{cocyc}$ be any vector decomposed into its boundary and cocycle parts.  The input-relative error bound of $\overline\Pi_{bd}$ is proportional to $\|x\| = \|x_{bd} + x_{cocyc}\|$, while for output-relative we need the bound to be proportional to $\|\Pi_{bd} x\| = \|x_{bd}\|$.  Therefore, a problem arises if $x_{cocyc}$ is much larger than $x_{bd}$; provided a bound on $\|x_{cocyc}\|/\|x_{bd}\|$, we can accordingly modify the accuracy of our projection operators $\T\Pi_{cbd}(\eps)$ and $\T\Pi_{hr}(\eps)$ to ensure $\overline\Pi_{bd}$ has small output-relative error for $x$.  Unfortunately, $\|x_{cocyc}\|/\|x_{bd}\|$ can be unbounded. To counteract this, we show that we can map $x$ to a different vector $x'$ before passing it to $\overline\Pi_{bd}$ such that (1) $x'$ has the same boundary component as $x$ (so $\Pi_{bd}\cdot x = \Pi_{bd}\cdot x'$), and (2) $\|x'_{cocyc}\|/\|x'_{bd}\|$ is bounded.

Specifically, our boundary projection operator is defined $\T\Pi_{bd} = (I-P_\Gamma)(I-P_T) \overline\Pi_{bd} (I-P_T)^T(I-P_\Gamma)^T$, defined based on two operators $P_T$ and $P_\Gamma$. The former was introduced by Cohen et al.~to obtain $\T\Pi_{cyc}$, and the latter is introduced in this paper; we sketch the ideas of both in this overview.  The operator $(I-P_T)^T(I-P_\Gamma)^T$ behaves as we need: it maps $x$ to a chain $x'$ with the same boundary component as $x$ and a relatively bounded cocycle part. We now describe $P_T$ and $P_\Gamma$.

Let $T$ be any spanning tree of the $1$-skeleton of $K$. $P_T$ is the operator that maps any $1$-chain to the unique $1$-chain with the same boundary in $T$. In particular, for any $1$-chain $x$, $(I-P_T)x$ is a cycle.

Next, let $\Gamma = \{\gamma_1, \ldots, \gamma_\beta\}$ be a $1$-homology basis in $K$.  $P_\Gamma$ is the operator that maps any $1$-cycle to the unique linear combination of $\Gamma$ that is in the same homology class. In particular, for any $1$-cycle $x$, $(I-P_\Gamma)x$ is a boundary.

Now let $F = (I-P_T)^T(I-P_\Gamma)^T$, so $F^T = (I-P_\Gamma)(I-P_T)$.
Consider any vector $x = x_{bd} + x_{cbd} + x_{hr}$.  We investigate what $F$ does to each of the three constituents of $x$; what can we say about $Fx_{bd}$, $Fx_{cbd}$ and $Fx_{hr}$? In what follows, we frequently use the fact that for any linear map $A$, $\ker(A)$ and $\im(A^T)$ orthogonally decompose the domain of $A$.

$(I-P_T)$ maps any $1$-chain to a cycle and $(I-P_\Gamma)$ maps any cycle to a boundary cycle; thus, $\im(F^T)$ is a subset of the boundary space.  It follows that $\ker(F)$ is a superset of the orthogonal complement of the boundary space, which is the cocycle space. So, $F$ maps any cocycle to zero, in particular, $Fx_{cbd} = 0$ and $Fx_{hr} = 0$. It remains to investigate $F x_{bd}$.

$P_T$ maps any cycle to zero, so $\ker(P_T)$ includes the cycle space; hence, $\im(P_T^T)$ is a subset of the orthogonal complement of the cycle space, which is the coboundary space. In particular, $\im(P_T^T)$ is a subset of the cocycle space. In addition, $P_\Gamma$ maps all boundary cycles to zero, so $\ker(P_\Gamma)$ includes the boundary space; hence, $\im(P_\Gamma^T)$ is within the orthogonal complement of the boundary space, which is the cocycle space.
Now consider
$$
Fx_{bd} = (I-P_T)^T(I-P_\Gamma)^Tx_{bd} = I x_{bd} - P_T^T(I-P_\Gamma)^Tx_{bd} - P_\Gamma^T x_{bd} = x_{bd} + x'_{cocyc},
$$ and observe that $x'_{cocyc}$ is indeed in the cocycle space as $\im(P_T^T)$ and $\im(P_\Gamma^T)$ are both within this space.

Overall, $Fx_{cbd} = 0$, $Fx_{hr} = 0$ and $F x_{bd} = x_{bd} + x'_{cocyc}$, so $F x$ has the same boundary part as $x$.  Moreover, the norm of the cocyclic part of $F x$, $\|x'_{cocyc}\|$, can now be bounded by $\|F\|\cdot \|x_{bd}\|$, as it is produced by applying $F$ to $x$.
The proof of Lemma 3.2 in Cohen et al.~and Lemma~\ref{lem:PGamma} and Corollary~\ref{cor:PGamma} of this paper provide a bound for $\|F\|$ that is dependent on the number of simplices of $X$, the smallest non-zero eigenvalue of the up-Laplacian of $X$, and the first Betti number of $K$. The accuracy and time complexity of the approximate boundary solver $\widetilde\Pi_{bd}$ are described in Lemma \ref{lem:boundary_form_cbd_harm} in the introduction.

\bibliography{main.bib}

\newpage
\appendix
\section{Harmonics of a Cohomology basis}
\label{sec:harm_of_cohom}

Recall $X$ is a collapsible complex embedded in $\R^3$, and $K\subset X$ is a subcomplex of $X$. Provided a homology basis for $K$, Black et al.~show how to compute a cohomology basis for $K$.  Black et al.~combine this result with Dey's algorithm for computing a homology basis (described in Lemma \ref{lem:tamal_hom_basis}) to compute a cohomology basis for any linearly embedded complex in $\R^3$.
\begin{lemma}[Dey~\cite{dey2019basis}]
\label{lem:tamal_hom_basis}
For a 2-dimensional simplicial complex $K$ linearly embedded in $\R^3$, there exists an algorithm computing a basis for $H_1(K, \R)$ in $O(n \log n + n \beta_1)$ time, where $n$ is the complexity of $K$.  Further, the basis is composed of vectors with all coordinates from $\{-1, 0, +1\}$.
\end{lemma}

In this paper, we use Black et al.'s cohomology basis to obtain a harmonic basis. In this section, we show that their cohomology basis has certain properties. First, we briefly sketch their algorithm (Section~\ref{subsec:Black_algorithm}).  Next, we show that the vectors in their basis have bounded lengths (Section~\ref{subsec:bounding_cohom_basis}).  Finally, we show that the harmonic projection of their cohomology basis results in harmonic vectors that are `sufficiently' independent (Section~\ref{subsec:hr_part_cohom_basis}), a property we need to obtain our harmonic projection operator.

\subsection{Black et al.~algorithm sketch}
\label{subsec:Black_algorithm}
Black et al.~\cite{CollapsibleUniverseBlack22} define an operator $C = C(X, K)$ that returns a cohomology basis if applied to a homology basis.
To that end, they compute an intermediate complex $T$ that is a maximal complex with the following two properties: (1) $K\subset T\subset X$, and (2) $H_2(T) = H_2(K)$.  Then, they define
\begin{equation}
\label{eqn:C}
C(X, K) = N(K,X)^T\circ F^T(X)\circ S^T(X,T)\circ \Pi_{C_2(T)^\perp}\circ S(X,T)\circ F(X) \circ N(K,X),
\end{equation}
where the operators $N(K, X)$, $F(X)$, $S(X,T)$ and $\Pi_{C_2(K)^\perp}$ are defined as follows:
\begin{itemize}
    \item $N(K,X): C_1(K)\rightarrow C_1(X)$, the Include operator, maps a $1$-chain in $K$ to the same chain in $X$.
    \item $F(X): C_1(X)\rightarrow C_2(X)$, the Fill operator, maps a $1$-cycle $\gamma$ in $X$ to a $2$-chain $x$ in $X$ such that $\partial_2[X]x = \gamma$.  To compute $x$, the Fill algorithm uses the collapsing sequence of $X$.  For each tetrahedron-triangle collapse $(\tau, t)$ the fill operator sets $x[t] = 0$.  For each triangle-edge collapse $(t, e)$, the value of $x[t]$ is determined by $\gamma[e]$, the Fill algorithm fixes this value, and recurse to compute the rest of $x$.
    \item $S(X, T): C_2(X)\rightarrow C_2(T)$, the Squeeze operator, maps a $2$-chain $x$ in $X$ to a $2$-chain $x'$ in $T$ with the same boundary, i.e.~$\partial_2[X]x = \partial_2[T]x'$.
    The Squeeze algorithm iteratively removes triangles in $X\backslash T$ and updates $x$ to keep its boundary invariant. In the end, the algorithm obtains $x'$ in $T$.

    Let $\{\sigma_1, \ldots, \sigma_k\}$ be the set of triangles in $X\backslash T$, and let $X_i = X\backslash\{\sigma_1, \ldots, \sigma_i\}$.  In particular, $X_0 = X$ and $X_k = T$.
    The Squeeze operator builds a sequence of $2$-chains $x = x_0, \ldots, x_k = x'$, such that for $0\leq i\leq k$, $x_i\in X_i$, and $\partial_2[X_i]x_i = \partial_2[X]x$.
    The $\sigma_i$'s are ordered so that each $\sigma_i$ is a face of a tetrahedron $\tau_i$ in $X_i$.  Such an order can be obtained via any graph traversal algorithm (e.g. BFS) on the dual graph of $X$ restricted to edges that are dual to $X\backslash T$ starting from the unbounded volume.  The traversal algorithm orders the edges based on their discovery time.  This order is equivalent to an order of the triangles of $X\backslash T$ with our desired property.

    At step $i$, the Squeeze algorithm removes $\sigma_i$, and modifies $x_i$ on the other faces of $\tau_i$ to obtain $x_{i+1}$ while ensuring $\partial_2[X_i]x_i = \partial_2[X_{i+1}]x_{i+1}$.

    \item $\Pi_{C_2(K)^\perp}:C_2(T)\rightarrow C_2(T)$ is the projection operator into the subspace of $2$-chains spanned by the simplices in $T\backslash K$.
\end{itemize}

The following lemma from Black et al.~summarizes the operator $C(X,K)$.

 \begin{lemma}[Black et al., Lemma 1.1 \cite{CollapsibleUniverseBlack22}]
\label{lem:black_et_al_C}
    Let $X$ be a collapsible simplicial complex in $\R^{3}$, and let $K\subset X$ be a subcomplex of $X$. Let $\beta$ be the rank of $H_1(K)$ and let $n$ be the total number of simplices of $X$. Let $\Gamma = \{\gamma_1,\ldots,\gamma_\beta\}$ be a homology basis for $K$. Let $C(X,K)$ be the operator described above. Then the set $C_\Gamma = \{C(X,K)\cdot\gamma_1,\ldots,C(X,K)\cdot\gamma_\beta\}$ is a cohomology basis for $K$. Furthermore, $C_\Gamma$ can be computed in $O(\beta\cdot n)$ time.
 \end{lemma}

\subsection{Bounding the lengths of the cohomology basis}
\label{subsec:bounding_cohom_basis}
The cohomology basis of Black et al.~is obtained by applying the operator $C = C(X, K)$ from Equation (\ref{eqn:C}) to the homology basis of Lemma~\ref{lem:tamal_hom_basis}.
The vectors of this homology basis have bounded length because all their coordinates are in $\{-1, 0, +1\}$.
In this section, we show that $C$ scales each cycle in a bounded way, hence, the length of the vectors in the cohomology basis is bounded. To that end, we bound the constituent operators of $C$ one by one.

\begin{lemma}
\label{lem:squeeze_norm_bound}
    The operator norm of the Squeeze operator is $\| S(X,T) \| \leq 2n_2$, where $n_2$ is the number of triangles in $X$.
\end{lemma}
\begin{proof}
    The $i$th iteration of the Squeeze operator distributes the value on a triangle $\sigma_i$ in $X_i$ over the boundary of an incident tetrahedron in $X_i$.  The triangle $\sigma_i$ is then removed from the complex, so this process can happen once per tetrahedron. As each triangle in $X$ is incident to at most two tetrahedra, the value of a triangle can be changed at most twice by the Squeeze operator.
    \par
    Let $x$ be the chain we are applying the Squeeze operator to, and let $x = x_0, x_1, \ldots, x_k = x'$ be the sequence of $2$-chains the squeeze algorithm obtains.  Let $\sigma\in X$ be any triangle. We use induction to show the following bounds for $x_j$, where $0\leq j\leq k$.
    \begin{enumerate}
        \item [(i)] $|x_j[\sigma]| = |x[\sigma]|$ if the value of $\sigma$ has never been changed.
        \item [(ii)] $|x_j[\sigma]| \leq |x[\sigma]| +  \sum_{i=1}^{j-1} |x[\sigma_i]|$ if the value of $\sigma$ has been changed once.
        \item [(iii)] $|x_j[\sigma]| \leq |x[\sigma]| + 2\sum_{i=1}^{j-1} |x[\sigma_i]|$ if the value of $\sigma$ has been changed twice.
    \end{enumerate}
    Note that the value of $\sigma$ changes at most twice during the algorithm, as each triangle is incident to at most two tetrahedra.
    These bounds hold in the beginning of the algorithm for $x$. We show that if they hold after the $(j-1)$st iteration (i.e. for $x_{j-1}$) they must hold after the $j$th iteration (i.e. for $x_j$.)
    Let $\sigma_j$ be the triangle removed at the $j$th iteration. The value of $\sigma_j$ has been changed at most one previous time during the algorithm, because $\sigma_j$ is incident to at most two tetrahedra in $X$, one of which is present in $X_j$. Thus, by the inductive hypothesis, we have
    \[
    |x_{j-1}[\sigma_j]| \leq |x[\sigma_j]| + \sum_{i=1}^{j-1}|x[\sigma_i]| = \sum_{i=1}^{j} |x[\sigma_i]|.
    \]
    Now, let $\sigma$ be any triangle that exists in the beginning of the $j$th iteration. If the value of $\sigma$ is unchanged during the $j$th iteration, the bounds will trivially hold. Otherwise, we have,
    \begin{align*}
    |x_j[\sigma]| &\leq |x_{j-1}[\sigma]| + |x_{j-1}[\sigma_j]|
    \leq
    |x_{j-1}[\sigma]| + \sum_{i=1}^{j} |x[\sigma_i]|.
    \end{align*}
    Using the induction hypothesis we obtain
    $
    |x_j[\sigma]| \leq
    |x[\sigma]| + \sum_{i=1}^{j} |x[\sigma_i]|
    $
    if $\sigma$ has not been changed before, and
    $
    |x_j[\sigma]| \leq
    |x[\sigma]| + 2\sum_{i=1}^{j} |x[\sigma_i]|
    $
    if it has been changed once before, as desired.

    Let $x' = S(X,K)x$ be the output of the Squeeze operator.  Our bounds imply that for each $\sigma\in K_2$,
    \[
    x'[\sigma] \leq 2\sum_{\eta\in X_2}{x[\eta]}= 2\|x\|_1.
    \]
    Using the inequality $\|y\|_1 \leq \sqrt{n}\|y\|_2$ that holds for any $n$-dimensional vector $y$, we obtain,
    \[
    \|x'\|_2  = \sqrt{\sum_{\sigma\in T_2} x'[\sigma]^2} \leq 2 \sqrt{n_2}\cdot \|x\|_1 \leq 2 n_2\cdot\|x\|_2,
    \]
    where $n_2$ is the number of triangles in $X$.
\end{proof}

\begin{lemma}
    \label{lem:fill_norm_bound}
    The operator norm of the Fill operator is $\|F(X)\| \leq 2(n_1+1)n_2/\sqrt{\lambda_{\min}(L^{up}_1[X])}$, where $\lambda_{\min}(L^{up}_1[X])$ is the smallest nonzero eigenvalue of $L^{up}_1[X]$, and $n_1$ and $n_2$ are the number of edges and triangles in $X$ respectively.
\end{lemma}
\begin{proof}
    We begin by showing that $\| F(X)\cdot \gamma \| \leq 2n_2/\sqrt{\lambda_{\min}(L^{up}_1[X])}\|\gamma\|$ for any 1-cycle $\gamma$.
    Since $X$ is collapsible, $H_1(X) = 0$. This implies $\gamma\in\im\boundary_2[X]$.  Next, observe that we can rearrange any collapsing sequence to ensure that all the tetrahedron-triangle collapses occur before other types of collapses \cite[Lemma 2.5]{OneLaplaciansCohen14}. Let $X'$ be the complex obtained after all tetrahedron-triangle collapses, and note that $H_2(X') = 0$. As $X'$ has no tetrahedra, then $\ker\boundary_{2}[X']=0$. Therefore, $\partial_2[X'] x = \gamma$ has a unique solution. The Fill operator returns this solution, $x = F(X)\cdot\gamma$, as it ignores all the triangles involved in tetrahedron-triangle collapses.
    \par
    Alternatively, to obtain this unique $x$, we can solve $\partial_2[X] z = \gamma$ for $z$ in $X$ and then squeeze $z$ to $X'$. Thus, $x = S(X, X') \partial^+_2[X]\cdot  \gamma$.  It follows that for any cycle $\gamma$,
    \[
    F(X)\cdot \gamma =  S(X, X') \partial^+_2[X]\cdot  \gamma.
    \]
    But $\|S(X, X')\| \leq 2n_2$ by Lemma~\ref{lem:squeeze_norm_bound}.  Further,
    \[
    \|\partial_2^+[X]\| \leq 1/(\sigma_{\min}(\partial_2[X])) =
    1/\sqrt{\lambda_{\min}(\partial_2[X]\partial^T_2[X])}
    = 1/\sqrt{\lambda_{\min}(L^{up}_1[X])},
    \]
    where $\sigma_{\min}$ is the smallest nonzero singular value and $\lambda_{\min}$ is the smallest nonzero eigenvalue. Therefore, we achieve a stronger bound than the lemma statement by combining the bounds on $\|S(X,X')\|$ and $\|\boundary_2^{+}[X]\|$. Note that this bound relies on $\gamma$ being in the image of $\partial_2$.
    \par
    Now consider an arbitrary 1-chain $y$. We will show there is a cycle $\gamma_y$ such that $\|\gamma_y\|\leq (n_1+1) \|y\|$ and $\|F(X)\gamma_y\| = \|F(X)y\|$. Observe that the 2-chain $F(x)y$ is uniquely determined by the values of $y$ on edges in triangle-edge collapses. Accordingly, to construct the cycle $\gamma_y$ from $y$, we will change the value of $y$ only on edges involved in edge-vertex collapses, as this will ensure that $\|F(X)\gamma_y\| = \|F(X)y\|$.
    \par
    The edges in edge-vertex collapses form a spanning tree $T$, so there is a unique chain $y_T$ on this spanning tree so that $\boundary y_T = \boundary y$. Therefore, we define $\gamma_y = y - y_T$. As $\|\gamma_y\| \leq \|y\| + \|y_T\|$, we can bound the size of $\|\gamma_y\|$ by bounding the size of $\|y_T\|$.
    \par
    The chain $y_T$ can be computed as follows. For each edge $e\in K_1$, let $T[e]$ be the 1-chain in $T$ with boundary $\boundary e$. In particular, if $e\in T$, $T[e]$ is $e$. The chain $T[e]$ is just the oriented path in $T$ between the endpoints of $e$, which is a chain with $\{-1,0,1\}$ coefficients. Thus, $y_T  = \sum_{\eta\in K_1}y[\eta]\cdot T[\eta]$.  We can therefore bound each coefficient of $y_T[e] \leq \sum_{\eta\in K_1} |y[\eta]|\leq \|y\|_1$,  which implies
    \[
    \| y_T \| = \sqrt{\sum_{e\in T} y^2_T[e]} \leq \sqrt{n_1}\cdot \|y\|_1 = n_1\cdot\|y\|
    \]
     as $\|y\|_1 \leq \sqrt{n_1}\|y\|$. Therefore, $\|\gamma_y\| \leq (n_1+1)\|y\|$. As $\|F(X)y\| = \| F(x)\gamma_y \|$, we conclude that $\|F(x)y\|\leq 2n_2/\sqrt{\lambda_{\min}(L^{up}_1[X])}\|\gamma_y\|\leq 2(n_1+1)n_2/\sqrt{\lambda_{\min}(L^{up}_1[X])}\|y\|$.
\end{proof}

\begin{lemma}
\label{lem:bound_Cgamma}
    There is a constant $\alpha$ such that
    $
    \|C(X, K)\| \leq \alpha\cdot n_1^2 n_2^4/\left(\lambda_{\min}(L^{up}_1[X])\right),
    $
    where $n_1$ and $n_2$ are the number of edges and triangles in $X$ respectively.
\end{lemma}
\begin{proof}
    The statement follows from (\ref{eqn:C}), the definition of $C$, as
    \[
    \|C\| \leq \|N(K,X)^T\|\cdot\|F(X)^T\|\cdot\|S(T,X)^T\|\cdot\|\Pi_{C_2(T)^\perp}\|\cdot\|S(T,X)\|\cdot\|F(X)\|\cdot\|N(K,X)\|,
    \]
    where $N(K, X)$ and $\Pi_{C_d(T)}$ are projection operators and their norm is bounded by 1.  Also, the norms of $S(X, T)$ and $F(X)$ are bounded by Lemma~\ref{lem:squeeze_norm_bound} and Lemma~\ref{lem:fill_norm_bound}, respectively.
\end{proof}

\subsection{The harmonic parts of the cohomology basis}
\label{subsec:hr_part_cohom_basis}
To obtain a harmonic basis, we project our cohomology basis into the cycle space.
Since we only have an approximate cycle projection operator, we need the harmonic part of our cohomology basis to be ``strongly'' independent so that the vectors stay independent after the approximate projection.
We say that a set of vectors are \EMPH{$\boldsymbol{\delta}$-linearly independent} if each one of them has distance at least $\delta$ from the span of the others.\footnotemark We show that the harmonic parts of a cohomology basis that is composed of integer vectors with bounded length must be strongly independent in this sense. We start with an auxiliary lemma to bound the determinant of a matrix with a totally unimodular submatrix.

\footnotetext{While it is not needed for this paper, it is worth noting that the minimum value of $\delta$ such that a set of vectors $V = \{v_1,\ldots,v_k\}$ is $\delta$-independent is tied to the smallest singular value $\sigma$ of the matrix $M_V$ that has $V$ for columns. Specifically, $\sigma \leq \delta \leq \sqrt{k} \sigma$. To see this, recall that $\sigma = \underset{x\neq 0}{\min}\, \sqrt{x^TM_V^TM_Vx / x^Tx}$. Suppose $\| v_1 - \sum_{i=2}^{k} x_i v_i \| = \delta$, then $\sigma \leq \| v_1 - \sum_{i=2}^{k} x_i v_i \| / \sqrt{x^Tx} \leq \delta$ where $x = [1\, x_2\cdots x_k$]. Alternatively, suppose $\sigma = \sqrt{x^TM_V^TM_Vx}$ for some $x$ with $\|x\|=1$. Then $x[i] \geq 1/\sqrt{k}$ for some $i$, and $\delta \leq \|v_i - \sum_{j\neq i} -(x[j]/x[i]) v_j\| = \sigma/x[i] \leq \sqrt{k}\sigma$.}

\begin{lemma}
\label{lem:det_bound_stacked_unimod}
Let $A$ be a totally unimodular matrix with $m$ rows and $n$ columns, and let $p_1, \ldots, p_k$ be $k = n-m$ row vectors of length $n$.  Let $B$ be the matrix obtained by stacking $p_1, \ldots, p_k$ on $A$.  We have $|\det(B)| \leq \prod_{i=1}^{k}{\|p_i\|_1}$
\end{lemma}
\begin{proof}
We use induction on $k$.  For $k = 0$, the statement is true as the matrix $B$ is totally unimodular itself.
To prove it for $k$, we use the Laplace expansion on the first row.
\begin{align*}
|\det(B)| &= \left|\sum_{j=1}^{n}{(-1)^{1+j}p_1[j]\det(B_{1, j})}\right|
\\ &\leq
\sum_{j=1}^{n}{|p_1(j)||\det(B_{1, j})|}
\\ &\leq
\left(\prod_{i=2}^{k}{\|p_i\|_1}\right)
\left(\sum_{j=1}^{n}{|p_1(j)|}\right) &&\text{(Induction Hypothesis)}
\\ &\leq
\left(\prod_{i=2}^{k}{\|p_i\|_1}\right)\cdot \|p_1\|_1 = \prod_{i=1}^{k}{\|p_i\|_1}
\end{align*}
\end{proof}

To show that a vector $v_1$ is far from the span of other vectors $\{v_2, \ldots, v_n\}$, it suffices to find a vector $w$ that is orthogonal to the span, but has a large inner product with $v_1$. Specifically, $w$ is a \EMPH{witness} for $v_1$ among $\{v_1, \ldots, v_k\}$ if
\begin{enumerate}
    \item [(i)] $w\cdot v_1 = 1$, and
    \item [(ii)] $w\cdot v_i = 0$ for all $i\in\{2, \ldots, k\}$.
\end{enumerate}
We show the existence of a witness gives a lower bound on the $\delta$-linear independence of $w$.

\begin{lemma}
\label{lem:witness_linear_comb_bound}
Let $\{v_1, \dots, v_k\}$ be a set of independent vectors, and let $w$ be a witness for $v_1$.
Then the distance between $v_1$ and $\spn(v_2, \ldots, v_k)$ is at least $1/\|w\|$
\end{lemma}
\begin{proof}
The dot product satisfies $u\cdot w = \cos\theta\|u\|\|w\|$, where $\theta$ is the angle between $u$ and $w$ for any vector $u$. In particular, this fact implies that $\| u \| \geq u\cdot w/\|w\|$ for any vector $u$.
Now let $\ell\in \spn(v_2, \ldots, v_k)$, so $\ell$ is a linear combination of $v_2, \ldots, v_k$. By property (ii), $w\cdot\ell = 0$. Therefore,
$
\| v_1 - \ell \| \geq (v_1-\ell)\cdot w/\|w\| = v_1\cdot w/\|w\| = 1/\|w\|.
$
\end{proof}

Now, we state our lemma showing that a 1-cohomology basis of bounded-length integer vectors have strongly linearly independent harmonic parts.

\deltaindependentharmonic

\begin{proof}
We show (i) and (ii) for $i=1$; the lemma follows for all other values of $i$ as we can reorder the $p_i$s and $h_i$s.

First, we show that $h_1$ is distance $\delta$ from any linear combination of $h_2, \ldots, h_\beta$, for a value of $\delta$ to be determined.
By Lemma~\ref{lem:witness_linear_comb_bound}, it suffices to show that a witness $b$ with the following properties exists.
\begin{enumerate}
    \item [(A)] $\|b\| \leq 1/\delta$,
    \item [(B)] $h_1 \cdot b = 1$, and
    \item [(C)] $h_j \cdot b = 0$ for any $2\leq j\leq \beta$.
\end{enumerate}
If we restrict $b$ to be a cycle, we can replace conditions (B) and (C) with the following equivalent conditions.
\begin{enumerate}
    \item [(B')] $b\cdot p_1 = 1$,
    \item [(C')] $b\cdot p_i = 0$ for all $2\leq i\leq \beta$, and
    \item [(D')] $b$ is a cycle.
\end{enumerate}
The reason conditions (B'), (C'), and (D') are equivalent to conditions (B) and (C) is if $b$ is a cycle and $p_i$ is a cocycle, thent $b\cdot p_i = b\cdot h_i$. This follows from the Hodge decomposition. As $p_i$ is a cocycle, then we can write $p_i = h_i + c_i$ where $h_i\in\ker L_1$ and $c_i\in\im\boundary_1^{T}$. However, $b\cdot c_i = 0$ as $b\in\ker\boundary_1$, so the claim follows.

We now show a cycle $b$ with properties (B') and (C') exists.
Let $b'$ be a harmonic cycle that is orthogonal to all $p_2, \ldots, p_\beta$; such a cycle $b'$ must exist as the space of harmonic cycles is $\beta$-dimensional.
Since $p_1, \ldots, p_\beta$ is a cohomology basis and $b'$ is not null-cohomologous, $b'$ cannot be orthogonal to $p_1$; thus, $b'\cdot p_1 = \alpha$ for some nonzero $\alpha$.  It follows that the rescaled cycle $b = b'/\alpha$ has properties (B') and (C').
Therefore, the following system of equations has a solution.  We now show it has a solution of length at most $1/\delta$ to complete the proof.
(Note that in the end we might obtain a solution different from $b$ as our solution need not be harmonic.)
\begin{equation}
\label{eqn:solve_for_b}
A b=\left[
\begin{array}{c}
   p_1 \\
   \vdots \\
   p_\beta \\
   \partial_1
\end{array}
\right] b =
\left[
\begin{array}{c}
   1 \\
   0 \\
   \vdots \\
   0
\end{array}
\right] = e_1.
\end{equation}
Here, $A$ is the matrix whose first $\beta$ rows are $p_1, \ldots, p_\beta$ and whose next $n_0$ rows are the $1$-boundary matrix $\partial_1$.
Note that the last block constraint $\partial_1 b = 0$ enforces $b$ to be a cycle.
Let $r$ be the rank of $A$, and let $\widetilde A$ be any $r\times r$ full rank submatrix of $A$. We solve Equation (\ref{eqn:solve_for_b}) by solving the following full rank system of equations
\begin{equation}
\label{eqn:solve_for_b_full_rank}
\widetilde A \widetilde b=\left[
\begin{array}{c}
   \widetilde p_1 \\
   \vdots \\
   \widetilde p_{\widetilde \beta} \\
   \widetilde \partial_1
\end{array}
\right] \widetilde b =
\widetilde{e}_1.
\end{equation}
This new system of equation is obtained by dropping rows and columns of $A$.  Whenever we drop a column we set a coordinate of $b$ to zero. In the end, we recover $b$ from $\widetilde b$ by appending some zeros, thus $\|\widetilde b\| = \|b\|$.
Whenever we drop a row we ignore the corresponding coordinate on the right side.  So, $\widetilde{e}_1$ will be a subvector of $e_1$. Note that all the $p_i$s are independent as they form a cohomology basis. Also, they are not in the image of $\partial_1^{T}$ as they are cohomology basis. Thus, $\widetilde A$ will have one row per $p_i$, and $\widetilde{e}_1$ is not zero.

We use Cramer's rule to solve Equation (\ref{eqn:solve_for_b_full_rank}).  Let $\widetilde b[i]$ be the $i$th coordinate of $\widetilde b$, and let $A_i$ be the matrix obtained by swapping the $i$th column of $A$ with $b$. We have
\[
\widetilde b[i] = \det(\widetilde A_i)/\det(\widetilde A).
\]
As $\widetilde A$ is a full rank integer matrix, we have
\[
|\det(\widetilde A)| \geq 1.
\]
On the other hand, note that the coordinates of $\T e_1$ are all zero, except the first coordinate which is $1$. By expanding
$\det(\widetilde A_i)$ on its $i$th column, we obtain
$|\det(\widetilde A_i)| = |\det(\widetilde B_i)|$, where $B_i$ is the $(r-1)\times(r-1)$ submatrix of $\widetilde A_i$ obtained by dropping the first row and $i$th column. To bound $|\det(\widetilde B_i)|$, note that $\widetilde B_i$ is composed of $\beta-1$ rows of subvectors of $p_i$s stacked on a submatrix of $\partial_1$, which is totally unimodular.  Hence, by Lemma~\ref{lem:det_bound_stacked_unimod},
\[
|\det(\widetilde A_i)| = |\det(\widetilde B_i)| \leq
\prod_{i=2}^{\beta}{\|p_i\|_1}.
\]
Putting things together,
\[
\widetilde b[i] = \frac{\det(\widetilde A_i) }{\det(\widetilde A)} \leq \prod_{i=2}^{\beta}{\|p_i\|_1},
\]
for each $1\leq i\leq r$. So,
\[
\|b\| \leq \|\widetilde b\| = \sqrt{\sum_{i=1}^{r}{\T b[i]^2}}
\leq \sqrt{r} \prod_{i=2}^{\beta}{\|p_i\|_1}
\leq \sqrt{n_1} \prod_{i=2}^{\beta}{\|p_i\|_1}
\leq \sqrt{n_1} \prod_{i=2}^{\beta}{\sqrt{n_1}\|p_i\|_2}
\leq (\sqrt{n_1})^{\beta}\cdot p_{\max}^{\beta-1}
\]
As $b$ is a witness for $h_1$, then $h_1$ has distance at least $\delta = 1/\|b\| = 1/\left(\sqrt{n_1})^{\beta}\cdot p_{\max}^{\beta-1}\right)$ from the span of $\{h_2, \ldots, h_\beta\}$.
In particular, $\|h_1\| \geq \delta$ as $0\in\spn\{h_2,\ldots,h_\beta\}$, which proves part (i) of the lemma.
In addition, $h_1/\|h_1\|$ has distance at least $\delta/\|h_1\|$ from the same span, but
\begin{align*}
\delta/\|h_1\| \geq \delta/p_{\max} =
1/(\sqrt{n_1}\cdot p_{\max})^{\beta}.
\end{align*}
This proves part (ii) of the lemma.
\end{proof}

We conclude this section with the following corollary of Lemma~\ref{lem:delta_indep_harm} about the cohomology $\{p_1, \ldots, p_\beta\}$ computed by applying the operator $C(X, K)$ (defined in Equation (\ref{eqn:C})) to the homology basis of Lemma~\ref{lem:tamal_hom_basis}.

\begin{corollary}
\label{cor:cohom_harmonic_prop}
Let $X$ be a collapsible complex embedded in $\R^3$ with a known collapsing sequence and let $K\subset X$ be a subcomplex of $X$. Let $\beta$ be the rank of $H_1(K)$ and let $n$ be the total number of simplices in $X$.
There is an $O(n\log n + \beta n)$ time algorithm for computing a cohomology basis $\{p_1, \ldots, p_\beta\}$ of $K$ with harmonic parts
$\{h_1, \ldots, h_\beta\}$ such that
\begin{enumerate}
    \item [(i)] each $h_i$ has length at least $\delta$, and
    \item [(ii)] the set $\{h_1/\|h_1\|, \ldots, h_\beta/\|h_\beta\|\}$ is $\delta$-linearly independent,
\end{enumerate}
where $\delta = (\lambda_{\min}(L_2^{up}(X))/(\alpha\cdot n_1^{3}n_2^4))^{\beta}$ for a constant $\alpha$.
\end{corollary}
\begin{proof}
Let $\Gamma = \{\gamma_1, \ldots, \gamma_\beta\}$ be the homology basis of Lemma~\ref{lem:tamal_hom_basis}.  Since the vectors in $\Gamma$ have coordinates in $\{-1, 0, +1\}$, the length of every $\gamma_i$ is bounded by $\sqrt{n_1}$.
Next, we apply $C(X,K)$ to $\Gamma$ to obtain a cohomology basis $P=\{C(X,K)\cdot\gamma_1,\ldots, C(X,K)\cdot\gamma_\beta\}=\{p_1, \ldots, p_\beta\}$. By Lemma~\ref{lem:bound_Cgamma}, every $p_i$ is an integer vector whose length is bounded by $\alpha\cdot n_1^{2.5}n_2^4/\lambda_{\min}(L_2^{up}(X))$.  So, by
Lemma~\ref{lem:delta_indep_harm}, we obtain parts (i) and (ii) of this lemma for $\delta = (\lambda_{\min}(L_2^{up}(X))/(\alpha\cdot n_1^{3}n_2^4))^{\beta}$. Lastly, the set $\Gamma$ can be computed in $O(n\log n)$ time by Lemma~\ref{lem:tamal_hom_basis} and $P$ can be computed in $O(\beta n)$ time by Lemma~\ref{lem:black_et_al_C}, which gives the running time for the algorithm.
\end{proof}

\section{Harmonic Projection}
\label{sec:harmonic_proj}

In the previous section, we showed that the harmonic parts of our cohomology basis are $\delta$-independent.
We exploit this property in this section to obtain our approximate harmonic projection operator.  To that end, we apply the approximate cycle projection $\T\Pi_{cyc}$ of Cohen et al.~to our cohomology basis to obtain an approximate harmonic basis $\T H$. Then, we use Gram-Schmidt on $\T H$ to obtain an orthonormal basis, which we use to obtain our projection operator.  In the rest of this section, we first analyze Gram-Schmidt applied to an approximate basis that is $\delta$-independent.  Then, we use this analysis to obtain our approximate harmonic projection operator $\T\Pi_{hr}$.

\subsection{Gram Schmidt; approximate subspace projection}

Let $N=\{\eta_1,\ldots,\eta_\beta\}$ be a set of linearly independent unit vectors that span a vector space $V$. The \EMPH{Gram-Schmidt Algorithm} takes as input $N$ and returns a set of \textit{orthonormal} vectors $G=\{g_1,\ldots,g_\beta\}$ that also span $V$. Gram-Schmidt iteratively constructs the vectors in $G$ such that for each $1\leq k\leq \beta$ (1) $g_k$ is a linear combination of $\{\eta_1,\ldots,\eta_k\}$, (2) $g_k$ is orthogonal to $\{g_1,\ldots,g_{k-1}\}$ and (3) $g_k$ has length one. These three facts guarantee the set $G$ is an orthonormal basis for the span of $H$. The set $G$ is computed as follows. Set $g_1 = \eta_1/\|\eta_1\|$. For $2\leq k\leq n$, compute an intermediate vector $u_k = \eta_k - \sum_{i=1}^{k} (\eta_k^Tg_i)g_i$, and then set $g_k = u_k / \|u_k\|$.
\par
For our problem, we have a set of unit vectors $N=\{\eta_1,\ldots,\eta_\beta\}$ that generate the 1-harmonic space and a set of unit vectors $\T N=\{\T \eta_1,\ldots,\T \eta_\beta\}$ such that $\| \eta_i - \T \eta_i\| \leq \eps\|\eta_i\|$ for each $1\leq i\leq\beta$. (We saw how $\Tilde N$ were constructed in Section \ref{sec:harm_of_cohom}.) We will run Gram-Schmidt on $\T N$ producing a set of orthonormal vectors $\T G = \{\T g_1,\ldots,\T g_\beta\}$. Our aim of this section is to show that $G$ and $\T G$ are close, and that the projection onto $G$ and $\T G$ is close. We begin our analysis with two helpful lemmas.

\begin{lemma}
\label{lem:distance_bounds_normalized}
Let $h$ and $\T h$ be two vectors.  We have
$
\left\| \frac{h}{\| h \|} - \frac{\T h}{\|\T h\|} \right\| \leq
\frac{2 \| h - \T h \|}{\|h\|}.
$
\end{lemma}
\begin{proof}
    We prove this directly.
    \begin{align*}
        \left\| \frac{h}{\|h\|} - \frac{\T h}{\| \T h \|}\right\|
        =& \frac{1}{\|h\|}\cdot \left\| h - \frac{\|h\|}{\|\T h\|}\T h \right\| \\
        = & \frac{1}{\|h\|}\cdot \left\| (h - \T h) + \left(1 - \frac{\|h\|}{\|\T h\|}\right)\T h \right\| \\
        \leq & \frac{1}{\|h\|}\cdot\left( \|h-\T h\| + \left|1 - \frac{\|h\|}{\|\T h\|}\right|\cdot\|\T h\| \right) &&\text{(Triangle Inequality)}\\
        = & \frac{1}{\|h\|}\cdot\left( \|h-\T h\| + \left|\|\T h\| - \|h\|\right| \right)  \\
        \leq& \frac{1}{\|h\|}\cdot(2\cdot\|h-\T h\|) && \text{(as }\left|\|\T h\| - \|h\|\right| \leq \| \T h - h \|\text{)}
    \end{align*}
\end{proof}

\begin{lemma}
\label{lem:approximate_inner_product_projection}
    Let $\eta,\T \eta, g, \T g$ be unit vectors such that $\| \eta - \T \eta\| < \eps_\eta$ and $\| g - \T g \| < \eps_g$ for $\eps_g,\eps_\eta<1$. The following statements are true:
    \begin{enumerate}
        \item[(1)] $| g^T\eta - \T g^T \T\eta | < \frac{3}{2}(\eps_g + \eps_\eta)$
        \item[(2)] $\| (g^T \eta)g - (\T g^T \T\eta)\T g \| < \frac{3}{2}\eps_g + \frac{5}{2}\eps_\eta$
    \end{enumerate}
\end{lemma}
\begin{proof}
    We first prove part (1). Define $e_g = g - \T g$ and $e_\eta = \eta - \T \eta$. We use this notation to rewrite $g^T \eta$ as
    $$
        g^T\eta = \T g^T \T \eta + \T g^T e_\eta + e_g^T \T \eta + e_g^T e_\eta.
    $$
    We subtract $\T g^T \T \eta$ from both sides and take the absolute value to get the bound
    \begin{align*}
        | g^T\eta - \T g^T \T \eta | &\leq |\T g^T e_\eta| + |e_g^T \T \eta| + |e_g^T e_\eta| \\
        &\leq \|g\|\|e_\eta\| + \|e_g\|\|\eta\| + \frac{1}{2}\|e_g\|\|e_\eta\| + \frac{1}{2}\|e_g\|\|e_\eta\| \tag{as $|u^Tv|\leq \|u\|\|v\|$} \\
        &= \|e_\eta\| + \|e_g\| + \frac{1}{2}\|e_g\|\|e_\eta\| + \frac{1}{2}\|e_g\|\|e_\eta\| \tag{as $\|\T g\|=\|\T\eta\|=1$} \\
        &\leq \frac{3}{2} \|e_\eta\| + \frac{3}{2} \|e_g\| \tag{as $\|e_\eta\|,\|e_g\| < 1$}\\
        &\leq \frac{3}{2}(\eps_\eta + \eps_g) \tag{by assumption}
    \end{align*}
    We now use Part (1) to prove Part (2).
    \begin{align*}
        \| (g^T \eta)g - (\T g^T \T \eta) \T g\| &= \| (g^T \eta)g - (\T g^T \T\eta) g + (\T g^T \T\eta) g- (\T g^T \T\eta) \T g\| \\
        &\leq | g^T\eta - \T g^T \T\eta |\cdot \|g\| + |\T g^T \T \eta|\cdot\| g - \T g \| \\
        &\leq \frac{3}{2}(\eps_\eta + \eps_g) + |\T g^T \T \eta|\cdot\| g - \T g \| \tag{by Part (1) and $\|g\|=1$} \\
        &\leq \frac{3}{2}(\eps_\eta + \eps_g) + \eps_g \tag{as $|\T g^T\T\eta|\leq \|\T g\|\cdot\|\T\eta\|=1$}\\
        &= \frac{3}{2}\eps_\eta + \frac{5}{2}\eps_g.
    \end{align*}
\end{proof}

\begin{lemma}
\label{lem:approximate_gram_schmidt}
    Let $0< \delta < 1$, and let $0<\eps < \left(\frac{\delta}{8\beta}\right)^{\beta}$. Let $N=\{\eta_1,\ldots,\eta_\beta\}$ be a set of $\delta$-linearly independent unit vectors, and let $\T N=\{\T \eta_1,\ldots,\T \eta_\beta\}$ be a set of unit vectors such that $\| \eta_i - \T \eta_i \| < \eps$. Let $G=\{g_1,\ldots,g_\beta\}$ be the output of running Gram-Schmidt on $N$, and let $\T G = \{\T g_1,\ldots,\T g_\beta\}$ be the output of running Gram-Schmidt on $\T N$. Then $\| g_i - \T g_i \| \leq \left(\frac{8\beta}{\delta}\right)^\beta\eps$ for each $i$.
\end{lemma}
\begin{proof}
    We will prove the stronger result that $\| g_k - \T g_k \| \leq \left(\frac{8k}{\delta}\right)^{k}\eps$ for each $1\leq k\leq\beta$. We prove this by induction. For $k=1$, we set $g_1=\eta_1$ and $\T g_1=\T \eta_1$; we have $\| \eta_1 - \T \eta_1 \| < \eps$ by assumption.
    \par
    Now suppose this is the case for $1\leq i\leq k.$ Observe that $\|g_i-\T g_i\|\leq\left( \frac{8k}{\delta}\right)^k\eps<1$ as $\eps<\left(\frac{\delta}{8\beta}\right)^\beta$ by assumption, so Lemma \ref{lem:approximate_inner_product_projection} applies. The result of the Gram-Schmidt algorithm before normalization are the vectors
    \begin{align*}
    u_{k+1} &= \eta_{k+1} - \sum_{i=1}^{k} (\eta_{k+1}^Tg_{i}) g_i,\: \text{and} \\
    \T u_{k+1} &= \T \eta_{k+1} - \sum_{i=1}^{k} (\T \eta_{k+1}^T \T g_{i}) \T g_i.
    \end{align*}
    If we compare the vectors $u_{k+1}$ and $\T u_{k+1}$, we find that
    \begin{align*}
        \| u_{k+1} - \T u_{k+1} \| &\leq \| \eta_{k+1} - \T \eta_{k+1} \| + \sum_{i=1}^{k} \left\| ( g_{i}^T\eta_{k+1}) g_i - (\T g_i^T \T{ \eta}_{k+1}) \T g_{i} \right\| &&\text{(Triangle Inequality)}\\
        &\leq \eps + \sum_{i=1}^{k} \left\| ( g_{i}^T\eta_{k+1}) g_i - (\T g_i^T \T{ \eta}_{k+1}) \T g_{i} \right\| &&\text{(By assumption)}\\
        &\leq \eps + \sum_{i=1}^{k} \left(\frac{3}{2}\eps + \frac{5}{2} \left(\frac{8i}{\delta}\right)^{i} \eps\right) &&\text{(Lemma~\ref{lem:approximate_inner_product_projection} Part (2))}\\
        &\leq \frac{3}{2}(k+1)\eps + \frac{5}{2}\eps \sum_{i=1}^{k}  \left(\frac{8i}{\delta}\right)^{i} \\
        &\leq \frac{3}{2}(k+1)\eps + \frac{5}{2}\eps\sum_{i=1}^{k}  \left(\frac{8k}{\delta}\right)^{k} \\
        &\leq \frac{3}{2}(k+1)\eps + \frac{5}{2}\frac{8^kk^{k+1}}{\delta^{k}}\eps \\
        &\leq \frac{3}{2}\frac{8^k(k+1)^{k+1}}{\delta^{k}}\eps + \frac{5}{2}\frac{8^k k^{k+1}}{\delta^{k}}\eps \tag{as $\delta<1$}\\
        &\leq 4\frac{8^{k}(k+1)^{k+1}}{\delta^{k}}\cdot \eps
    \end{align*}
    This bound is on difference between the unnormalized vectors $u_{k+1}$ and $\T u_{k+1}$. To bound the distance between the normalized vectors $g_{k+1}$ and $\T g_{k+1}$, we first observe that $\| g_{k+1} \|>\delta$. This follow from the $\delta$-independence of $N$ as $g_{k+1}$ is defined as $g_{k+1}=\eta_{k+1} - \sum_{i=1}^{k} (e_i^T \eta_i)e_i$ and $\sum_{i=1}^k (g_i^T \eta_i)g_i$ is a linear combination of $N\setminus\{n_{k+1}\}$. This implies
    \[
    \| g_{k+1} - \T{g}_{k+1}\| = \left\| \frac{u_{k+1}}{\|u_{k+1}\|} - \frac{\T{u}_{k+1}}{\|\T{u}_{k+1}\|}\right\| \leq
    \frac{2\| u_{k+1} - \T u_{k+1} \|}{\|u_{k+1}\|} \leq
    \frac{2}{\delta}\cdot\left(4\frac{8^{k}(k+1)^{k+1}}{\delta^{k}}\right) \eps,
    \]
    as desired, where the first inequality follows from Lemma~\ref{lem:distance_bounds_normalized}.
\end{proof}

If $G=\{g_1,\ldots,g_\beta\}$ is a set of orthonormal vectors, recall that the orthogonal projection onto $\spn G$ is the linear operator $\Pi_{\spn G} = \sum_{i=1}^{\beta} g_ig_i^T$. If there is a set of orthonormal vectors $\T G = \{\T g_1,\ldots,\T g_\beta\}$ such that $\|g_i - \T g_i\|<\eps$, then the following lemma shows that $\|\Pi_{\spn G} - \Pi_{\spn\T{G}}\|$ is bounded.

\begin{lemma}
\label{lem:approximate_projection}
    Let $\{g_1,\ldots,g_\beta\}$ be a set of orthonormal vectors, and let $\{\T g_1,\ldots, \T g_\beta\}$ be a set of orthonormal vectors such that $\| g_i - \T g_i \| < \eps$ for each $i$. We have $\| \Pi_{\spn G} - \Pi_{\spn\T{G}} \| = \left\| \sum_{i=1}^{\beta} g_ig_i^T - \T g_i\T g_i^T \right\| < 2\cdot \beta\cdot \eps$.
\end{lemma}
\begin{proof}
    Let $v$ be any vector. We use the triangle inequality two times in the following calculation.
    \begin{align*}
        \left\| \sum_{i=1}^{\beta}(g_i^Tv)g_i - (\T g_i^T v)\T g_i \right\| &\leq \sum_{i=1}^{\beta} \| (g_i^Tv)g_i - (\T g_i^T v)\T g_i \| \\
        &= \sum_{i=1}^{\beta} \| (g_i^Tv)g_i - (\T g_i^T v) g_i + (\T g_i^T v) g_i - (\T g_i^T v)\T g_i \| \\
        &= \sum_{i=1}^{\beta} |g_i^Tv - \T g_i^T v|\cdot \| g_i \| + |\T g_i^T v|\cdot\| g_i - \T g_i \|
    \end{align*}
    But, $\| g_i - \T g_i \| < \eps$, $|\T g_i^T v|\leq \|v\|$,  $\| g_i \| = 1$, and $|g_i^Tv - \T g_i^T v| \leq \|g_i^T - \T g_i^T\| \cdot\|v\|\leq \eps\|v\|$.  Hence,
    \begin{align*}
        \| \sum_{i=1}^{\beta}(g_i^Tv)g_i - (\T g_i^T v)\T g_i \| \leq
         \sum_{i=1}^{\beta} 2\cdot\eps\cdot\|v\|
        =  2\cdot \beta\cdot\eps\cdot\|v\|,
    \end{align*}
    and the proof is complete.
\end{proof}

The main result of this section follows from
Lemma~\ref{lem:approximate_projection} and
Lemma~\ref{lem:approximate_gram_schmidt}.

\approximategramschmidt

\subsection{The harmonic projection operator}
Now, we are ready to describe our harmonic projection operator $\T\Pi_{hr}$.
To obtain this operator, we use the following projection operator into the cycle space.

\begin{lemma}[Cohen et al.~\cite{OneLaplaciansCohen14}, Lemma 3.2]
\label{lem:graph_projection}
Let $K$ be a simplical complex with $n = n_1 + n_0$ total number of edges and vertices, and $\varepsilon >0$. In $\widetilde{O}(n \log n \log (n/\varepsilon))$ time, we can compute symmetric matrices $\widetilde{\Pi}_{cbd}(\eps)$ and $\widetilde{\Pi}_{cyc}(\eps)$\footnotemark such that
\begin{eqnarray}
&& (1-\varepsilon)\Pi_{cbd} \preceq \widetilde{\Pi}_{cbd}(\eps) \preceq \Pi_{cbd} \\
&& (1-\varepsilon)\Pi_{cyc} \preceq \widetilde{\Pi}_{cyc}(\eps) \preceq \Pi_{cyc}
\end{eqnarray}
Moreover, for any $1$-chain $x$, $\widetilde{\Pi}_{cbd}(\eps) \cdot x$ and  $\widetilde{\Pi}_{cyc}(\eps)\cdot x$ can be computed in the same asymptotic running time.
\end{lemma}

\footnotetext{Cohen et al.~describe the operator $\widetilde{\Pi}_{cyc}$ as being an approximate projection onto $\im\boundary_2$, \textit{not} $\ker\boundary_1$ as we describe it. The complex $K$ in Cohen et al.~satisfies $\ker\boundary_1=\im\boundary_2$, so in their setting, the two operators are equivalent. However, inspecting their algorithm reveals that when $\ker\boundary_1\neq \im\boundary_2$, their algorithm actually approximates $\Pi_{cyc}$ and not $\Pi_{bd}$.}

We find the following two basic lemmas useful when analyzing properties of our harmonic projection operator.

\begin{lemma}
\label{lem:matrix_norm_to_loewner_bound}
If $A$ is a symmetric matrix, then
$\| A \| \leq \eps \Longleftrightarrow -\eps I \preceq A \preceq \eps I$.
\end{lemma}
\begin{proof}
    Let $\lambda_1 \geq \cdots \geq \lambda_n$ be the eigenvalues of $A$. Note that $\lambda_1 x^Tx \geq x^T A x \geq \lambda_n x^Tx$ for any vector $x$.
    \par
    \underline{($\Rightarrow$)} Suppose $\| A \| \leq \eps$. Then $\frac{\| Ax \|}{\| x \|} \leq \eps$ for any vector $x$. In particular, if $v_i$ is the eigenvector for $\lambda_i$,
    \begin{align*}
    \eps \geq \frac{\| Av_i \|}{\|v_i\|} = \frac{\sqrt{ (A v_i)^T A v_i }}{\sqrt{v_i^T v_i}}
    = \frac{\sqrt{ \lambda_i^2 v_i^T v_i }}{\sqrt{v_i^T v_i}}
    = | \lambda_i |
    \end{align*}
    Since all eigenvalues of $A$ have absolute value at most $\eps$, $0\preceq A+\eps I$ and $A-\eps I\preceq 0$, as desired.
    \par
    \underline{($\Leftarrow$)} Now suppose that $-\eps I \prec A \prec \eps I$. Then $-\eps \leq x^T A x / x^Tx \leq \eps$ for all vectors $x$. By the Courant-Fischer Theorem, this implies that all eigenvalues of $A$ are in $[-\eps, \eps]$. Also, we have $\sigma_{\max}(A) = |\lambda_{\max}(A)|$ as $A$ is symmetric.  Thus, $\|A\| = \sigma_{\max}(A) = |\lambda_{\max}(A)| \leq \eps$.
\end{proof}

\begin{lemma}
\label{lem:lowener_to_matrix_output_rel}
    Let $\Pi$ be an orthogonal projection operator, and let $\T\Pi$ be a symmetric operator such that $(1-\eps)\Pi \preceq \T\Pi \preceq (1+\eps)\Pi$. Let $p$ be any vector, and let $h = \Pi p$ and $\T h = \T\Pi p$. Then $\| h - \T h \| < \eps \| h \|$.
\end{lemma}
\begin{proof}
    First observe that $-\eps\Pi \preceq \T\Pi-\Pi \preceq \eps\Pi$; this observation and Lemma \ref{lem:matrix_norm_to_loewner_bound} imply that $\| \T\Pi - \Pi \| < \eps$ as $\Pi \preceq I$. This gives the loose bound of $\| h_i - \T h_i \| \leq \eps \| p \|$. To get the tighter bound, we also need the observation that $\im\T\Pi \subset \im\Pi$, which follows from the assumption of the lemma. Therefore, $\T\Pi = \T\Pi\circ\Pi$ and $\Pi - \T\Pi = (\Pi - \T\Pi)\circ \Pi$. We therefore know that $\| (\Pi - \T\Pi) p \| = \| (\Pi - \T\Pi)\Pi p \| \leq \eps\|\Pi p\|$, or in the given notation, $\| h - \T h\| \leq \eps\|h\|$.
\end{proof}

Now we are ready to present our harmonic projection operator.

\harmonicprojection*

\begin{proof}
Let $\{p_1, \ldots,p_\beta\}$ be the cohomology basis of Corollary~\ref{cor:cohom_harmonic_prop}, and let $p_{\max}$ be the length of the longest vector of this basis.
First, we compute an approximate harmonic basis,
$\{\T h_1/\|\T h_1\|, \ldots, \T h_\beta/\|\T h_\beta\|\}$, where
$\T h_i = \widetilde\Pi_{cyc}(\eps')(p_i)$, for all $1\leq i\leq \beta$ and $\widetilde\Pi_{cyc}(\eps')$ is the operator of
Lemma~\ref{lem:graph_projection}, for some $\eps'$ to be determined. By Lemma~\ref{lem:lowener_to_matrix_output_rel}, we have
\begin{align}
    -\eps' \Pi_{cyc} \preceq \widetilde{\Pi}_{cyc}(\eps') - \Pi_{cyc} \preceq \eps' \Pi_{cyc}
    &\Rightarrow
    \|(\widetilde{\Pi}_{cyc}(\eps') - \Pi_{cyc})p_i\| \leq \eps'\cdot \|h_i\| \nonumber \\
    &\Rightarrow
    \|\widetilde{h}_i - h_i\| \leq \eps' \cdot \|h_i\|  \nonumber \\
    &\Rightarrow
    \|\widetilde{h}_i/\|\widetilde{h}_i\| - h_i/\|h_i\|\| \leq 2\eps' &&\text{(Lemma~\ref{lem:distance_bounds_normalized})}
    \label{eqn:eps_h_hprim}
\end{align}

Next, we apply Gram Schmidt to $\{\T h_1/\|\T h_1\|, \ldots, \T h_\beta/\|\T h_\beta\|\}$ to obtain $\{\T g_1, \ldots, \T g_\beta\}$.  Let $\{g_1, \ldots, g_\beta\}$ be the result of applying Gram Schmidt to the exact harmonic basis $\{h_1/\|h_1\|, \ldots, h_\beta/\|h_\beta\|\}$.  By Lemma~\ref{lem:approximate_gram_schmidt},
the fact that
$\{h_1/\|h_1\|, \ldots, h_\beta/\|h_\beta\|\}$ is $\delta$-independent from Corollary~\ref{cor:cohom_harmonic_prop}-(ii)),
and the bound in Equation (\ref{eqn:eps_h_hprim}), we have
$
\|g_i - \T g_i\| \leq \left(\frac{8\beta}{\delta}\right)^\beta\eps'.
$
We define our projection operator as
\[
\widetilde\Pi_{hr}(\varepsilon) = \sum_{i=1}^{\beta} \T g_i\T g_i^{T}.
\]
By combining Corollary~\ref{cor:approximate_gram_schmidt_operator}, the fact that
$\{h_1/\|h_1\|, \ldots, h_\beta/\|h_\beta\|\}$ is $\delta$-independent, and the approximation bound of Equation (\ref{eqn:eps_h_hprim}), we have
\[
\|\widetilde\Pi_{hr}(\eps) - \Pi_{hr}\| \leq
2\cdot \beta\cdot \left(\frac{8\beta}{\delta}\right)^\beta
 \cdot \eps'.
\]
By setting $\eps' = (\delta/8\beta)^{\beta+1}\cdot \eps$, we obtain $\|g_i - \T g_i\| \leq \eps$ for all $1\leq i \leq \beta$, which proves part (i) of the lemma. Furthermore, $\|\widetilde\Pi_{hr}(\eps) - \Pi_{hr}\| \leq \eps$, which is equivalent to part (ii) of the lemma by Lemma~\ref{lem:matrix_norm_to_loewner_bound}.
Substituting the value of $\delta$ from Corollary~\ref{cor:cohom_harmonic_prop}, we obtain
\[
\eps' = \left(\frac{\lambda_{\min}}{\beta\cdot n_1 n_2}\right)^{\Omega(\beta^2)}
\Rightarrow
\log\frac{1}{\eps'} = O\left(\beta^2\log \frac{\beta\cdot n_1 n_2}{\lambda_{\min}}\right)
= O\left(\beta^2\log \frac{n_1 n_2}{\lambda_{\min}}\right),
\]
where the last inequality is because $\beta\leq n_1$.

By Corollary~\ref{cor:cohom_harmonic_prop}, $\{p_1, \ldots, p_\beta\}$
can be computed in $O(n\log n + \beta n)$ time.
By Lemma~\ref{lem:graph_projection}, $\widetilde\Pi_{cyc}(\eps')$ can be applied in
$\tilde{O}(\beta^2 \cdot n \cdot \log n \cdot \log(\frac{n}{\lambda_{min}(X)\cdot \eps}))$ time.
We can normalize each $\T h_i$ in $O(\beta n_1)$ time.
Finally, we can apply Gram Schmidt to $\{\T h_1/\|\T h_1\|, \ldots, \T h_\beta/\|\T h_\beta\|\}$ in $O(\beta^2 n_1)$ time.
Therefore, $\{g_1 \ldots g_\beta\}$ can be computed in
$\tilde{O}(\beta^2 \cdot n \cdot \log n \cdot \log(\frac{n}{\lambda_{min}(X)\cdot \eps}))$ time.
\par
Constructing the basis $\{g_1,\ldots,g_\beta\}$ is also the limiting step for computing the projection $\widetilde\Pi_{hr}(\eps)\cdot x$ for a vector $x$. Once the basis is built, the projection $\widetilde\Pi_{hr}(\eps)\cdot x$ can be computed in $O(\beta n_1)$ time; this is because computing each product $g_ig_i^{T}x$ takes $O(n_1)$ time. Therefore, the projection $\widetilde\Pi_{hr}(\eps)\cdot x$ can be computed in the same asymptotic running time as computing the basis.
\end{proof}

\section{Boundary Projection}
\label{sec:boundary_proj}
In this section, we build a projection operator into the boundary space based on the operators $\T\Pi_{hr}$ and $\T\Pi_{cbd}$ of Lemma~\ref{lem:harmonic_projection} and Lemma~\ref{lem:graph_projection}, respectively.

The natural candidate for an approximate boundary projection is $\overline\Pi_{bd} = I - \T\Pi_{hr} -\T\Pi_{cbd}$ as the exact boundary projection $\Pi_{bd} = I - \Pi_{hr} - \Pi_{cbd}$ by the Hodge Decomposition.  It is straight forward to see that $\overline\Pi_{bd}$ has bounded input-relative error, that is $-\eps I \preceq \Pi_{bd} - \overline\Pi_{bd} \preceq \eps I$.  However, we need a projection into the boundary space with output-relative error for our solver. We instead define our approximate boundary projection as $\T\Pi_{bd} = (I-P_\Gamma)(I-P_T)\overline \Pi_{bd}(I-P_T)^T(I-P_\Gamma)^T$ using operators $P_T$ and $P_\Gamma$ we introduce in this section.
\par
In this section, we first define $P_T$ and $P_\Gamma$ and prove some useful properties about them. Then we show that $\T\Pi_{bd}$ is a projection operator into the boundary space with output-relative error bounds.

\subsection{The helper operators}

Cohen et al.~\cite{CohenLogSqrtSolver14} describe the operator $P_T$ in the proof of Lemma 3.2 in their paper. Given a spanning tree $T$, $P_T$ maps a $1$-chain in $K$ to a $1$-chain in $T$ with the same boundary, i.e. $\boundary_1\Pi_T x = \boundary_1 x$ for all $x$.

We introduce the new operator $P_\Gamma$. Given a fixed homology basis $\Gamma$, $P_\Gamma$ is the linear map that maps each $1$-cycle to a homologous $1$-cycle that is a linear combination of the cycles in the basis $\Gamma$.

We bound the output-relative error of $\T\Pi_{bd} = (I-P_\Gamma)(I-P_T)\overline \Pi_{bd}(I-P_T)^T(I-P_\Gamma)^T$ in this section.  The following lemma captures two key properties we need to that end.

\begin{lemma}
\label{lem:PTPGammaProp}
Let $P_T$ and $P_\Gamma$ be as defined above.  We have
\begin{enumerate}
    \item [(i)] $(I - P_{\Gamma})(I- P_T) \Pi_{bd} = \Pi_{bd}$, and
    \item [(ii)] $\Pi_{bd} (I - P_{\Gamma})(I - P_T) = (I - P_{\Gamma})(I - P_T)$.
\end{enumerate}
\end{lemma}
\begin{proof}
    First, we prove (i).  $\Pi_{bd}x$ returns the boundary part $x_{bd}$ of a chain $x$. As the boundary of a boundary chain is zero, then
    $
    P_T\Pi_{bd}x = 0 \Rightarrow (I-P_T)\Pi_{bd} = \Pi_{bd}.
    $
    As well, as $x_{bd}$ is a boundary, the only linear combination of the set $\Gamma$ homologous to $x_{bd}$ is 0, so $P_\Gamma\Pi_{bd}x = 0\Rightarrow (I - P_{\Gamma})\Pi_{bd} = \Pi_{bd}$. Together, these results imply statement (i).

    Next, we prove (ii).
    Let $x$ be any one chain.  Since $P_T\cdot v$ is a chain with the same boundary as $x$, $(I-P_T)x$ is a cycle.
    By the definition of $P_\Gamma$, $P_\Gamma(I-P_T)x$ is homologous to $(I-P_T)x$, so $(I-P_\Gamma)(I-P_T)x$ is a boundary. Thus, $\Pi_{bd} (I-P_\Gamma)(I-P_T)x = (I-P_\Gamma)(I-P_T)x$
\end{proof}

In addition, to the lemma above, we need to bound the norm of $(I-P_T)(I-P_T)^T$, and $(I-P_\Gamma)(I-P_\Gamma)^T$, which we do in the following section.

\subsubsection{Bounding the operators}

Cohen et al.~describe a bound on the norm of $(I-P_T)(I-P_T)^T$ that we formalize in the following lemma.

\begin{lemma} [Cohen et al., Proof of Lemma 3.2~\cite{OneLaplaciansCohen14}]
\label{lem:PT}
Let $K$ be any simplicial complex, and let
$T$ be a spanning tree of the $1$-skeleton of $K$.  Let $P_T$ be the operator that maps any 1-chain $x$ to the unique 1-chain on $T$ with the same boundary, that is (i) $P_T\cdot x\in C_1(T)$, and (ii) $\partial_1 x = \partial_1 P_T\cdot x$. We have $(I - P_T)(I-P_T)^T \preceq n_1^2 I$, where $n_1$ is the number of edges of $K$.
Further,
for any $x$, $P_T\cdot x$ can be computed in $O(n_1)$ time.
\end{lemma}
\begin{proof}
Note that $P_T$ maps each edge to a simple path, hence all elements of $P_T$ are in $\{0, 1, -1\}$.  Further, the diagonal elements of $P_T$ are non-negative; if an edge is on $T$ then $P_T$ map it to itself, otherwise, it maps it to a path that does not contain the edge. It follows that all elements of $I - P_T$ have absolute value at most one.
Therefore, by Lemma~\ref{lem:max_eig_bound_row_col}, $\|I-P_T\|\leq n_1$, hence
\[
(I - P_T)(I-P_T)^T \preceq n_1^2 I
\]

To compute $P_T\cdot x$, we first compute the required boundary $d = \partial_1\cdot x$ in time proportional to the total number of edges in $K$, which is $O(n_1)$.
For each leaf $\ell$ of $T$, $d$ determines the required flow on its only incident edge $e_\ell$.  We assign this flow to $e_\ell$, update the value of $d$ on the other incident edge to $e_\ell$, and recurse to the tree $T\setminus\{\ell,e_\ell\}$. We spend $O(1)$ time per recursion, so the total running time will be $O(n_1)$.
\end{proof}

Next, we bound the norm of $(I-P_\Gamma)(I-P_\Gamma)^T$. To that end, we need a couple of auxiliary lemmas about the $2$-norm of matrices.

\begin{lemma}
\label{lem:max_eig_bound_row_col}
Let $A$ be a $n_{row}\times n_{col}$ matrix. We have:
(i) $\|A\| \leq \sqrt{n_{row}}\cdot row_{\max}$, and
(ii) $\|A\| \leq \sqrt{n_{col}}\cdot col_{\max}$,
where $row_{max}$, and $col_{\max}$ are the maximum Euclidean norm of the rows and columns of $A$, respectively.
\end{lemma}
\begin{proof}
First, we show (i); (ii) follows as $\|A\| = \|A^T\|$. Recall $\|A\| = \sup\{\|Ax\| : \|x\| = 1\}$. Let $x$ be a unit vector. Then
\begin{align*}
\|A\cdot x\| =
\left|\left|\left[
\begin{array}{c}
   row_1(A) \\
   \vdots \\
   row_{n_{row}}(A)
\end{array}
\right]\cdot x\right|\right|
\leq
\left|\left|\left[
\begin{array}{c}
   \|row_1(A)\| \\
   \vdots \\
   \|row_{n_{row}}(A)\|
\end{array}
\right]\right|\right| \leq
\sqrt{n_{row}}\cdot row_{\max}
\end{align*}
as $\|row_i(A)\cdot x\| \leq \|row_i(A)\|\cdot\|x\|=\|row_i(A)\|$. Hence, we achieve the bound on $\|A\|$.
\end{proof}

\begin{lemma}
\label{lem:eigen_lower_based_abs_values}
Let $A$ be a $n\times n$ full rank integer matrix.  We have
\[
\|A^{-1}\| \leq n^{\frac{n}{2}+1}\cdot A^n_{\max},
\]
where $A_{\max}$ is the maximum absolute value of a any element of $A$.
\end{lemma}
\begin{proof}
Let $A: V \rightarrow U$. (Technically, $V=U=\R^{n}$, but using $V$ and $U$ helps distinguish the domain and codomain.) Since $A$ is full rank (and in particular surjective),
\[
\|A^{-1}\| = \sup\{\|A^{-1}u\|/\|u\|: u\in U\} = \sup\{\|A^{-1}Av\|/\|Av\|: v\in V\} = \sup\{\|v\|/\|Av\|: v\in V\}.
\]
Note that $\|v\|/\|A\cdot v\|$ does not change by rescaling $v$.  So, we can assume that the first coordinate $v[1] = 1$ and that $v[1]$ is a coordinate with maximum absolute value (after permuting coordinates of $v$ and columns of $A$.) So, $\|v\| \leq \sqrt{n}$.  We have
\[
A\cdot v = col_1(A) + \sum_{i=2}^{n}{v_i\cdot col_i(A)}.
\]

Let $w$ be the unique solution of the equation
\[
A^T\cdot w = e_1.
\]
By the definition, $w$ is a witness vector of $col_1(A)$ among $\{col_2(A),\ldots, col_n(A)\}$.  By Lemma~\ref{lem:witness_linear_comb_bound},
\[
\|A\cdot v\| \geq 1/\|w\| \Rightarrow
\frac{\|v\|}{\|A\cdot v\|} \leq \|v\|\cdot\|w\|
\]
Using Cramer's rule,
\[
w[i] = \frac{\det(A_i)}{\det(A)},
\]
where $A_i$ is the matrix obtained by replacing column $i$ of $A$ with $e_1$. Since $A$ is a full rank integer matrix $|\det(A)| \geq 1$.  Since $A_i$ is a matrix with max value $A_{\max}$, each of its columns has length at most $\sqrt{n}\cdot A_{\max}$.  So $|\det(A_i)| \leq (\sqrt{n}\cdot A_{\max})^n$. Thus,
\[
|w[i]| \leq (\sqrt{n}\cdot A_{\max})^n \Rightarrow \|w\| \leq \sqrt{n}\cdot(\sqrt{n}\cdot A_{\max})^n = n^{(n+1)/2}A^n_{\max}.
\]
Therefore,
\[
\frac{\|v\|}{\|A\cdot v\|}
\leq
\|v\|\cdot\left(n^{(n+1)/2}A^n_{\max}\right)
\leq \sqrt{n}\cdot n^{(n+1)/2}A^n_{\max}
= n^{n/2+1}A^n_{\max}.
\]
The second inequality follows from the fact that $\|v\|\leq \sqrt{n}$.
\end{proof}

Now, we are ready to bound $(I-P_\Gamma)(I-P_\Gamma^T)$.

\begin{lemma}
\label{lem:PGamma}
Let $\Gamma = \{\gamma_1, \ldots, \gamma_\beta\}$ be any homology basis, and let $P_{\Gamma}$ be the linear operator that for any cycle $\alpha$, returns the unique linear combination of the cycles of $\Gamma$ that is homologous to $\alpha$.
Also, let $P = \{p_1, \ldots, p_\beta\}$ be any cohomology basis.
We have
\[
    (1 - P_{\Gamma})(1 - P_{\Gamma})^T \preceq
    (1 + \beta^{(\beta+3)/2}\cdot (p_{\max}\cdot\gamma_{\max})^{\beta+1})^2
    \cdot I,
\]
where $p_{\max}$ and $\gamma_{\max}$ are the maximum Euclidean norm of the vectors of $P$ and $\Gamma$ respectively.  Further,
for any $v$, $P_{\Gamma}\cdot v$ can be computed in $O(\beta^2 n_1 + \beta^\omega)$ time provided $P$ and $\Gamma$.
\end{lemma}
\begin{proof}
We implement $P_{\Gamma}$ as follows.
For any cycle $v$, we need $P_\Gamma v = v' = x_1\gamma_1 + \ldots, x_\beta\gamma_\beta$ to be homologous to $v$.  Equivalently by Fact \ref{fact:homology_annotation} in Section \ref{sec:background}, we need $v$ and $v'$ to have the same inner product with each $p_i$. We can summarize this in matrix notation as
\[
\left[
\begin{array}{c}
   p_1^T \\
   \vdots \\
   p_\beta^T
\end{array}
\right]\cdot v =
\left[
\begin{array}{c}
   p_1^T \\
   \vdots \\
   p_\beta^T
\end{array}
\right] \cdot v' =
\left[
\begin{array}{c}
   p_1^T \\
   \vdots \\
   p_\beta^T
\end{array}
\right]\cdot\left[
\gamma_1 \ldots \gamma_\beta
\right]\cdot x,
\]
where $x = [x_1 \ldots x_\beta]^T$. Let
\[
M = \left[
\begin{array}{c}
   p_1^T \\
   \vdots \\
   p_\beta^T
\end{array}
\right]\cdot\left[
\gamma_1 \ldots \gamma_\beta
\right].
\]
The matrix $M$ is a $\beta\times\beta$ full rank matrix as both $\{p_1 \ldots, p_\beta\}$ and $\{\gamma_1, \ldots, \gamma_\beta\}$ are linearly independent, so $M^{-1}$ exists. We have
\[
P_{\Gamma}\cdot v =
v' =
[\gamma_1 \ldots \gamma_\beta]\cdot x =
[\gamma_1 \ldots \gamma_\beta]\cdot M^{-1}\cdot
\left[
\begin{array}{c}
   p_1^T \\
   \vdots \\
   p_\beta^T
\end{array}
\right]\cdot v.
\]
Provided $P$ and $\Gamma$, we can compute $M$ in $O(\beta^2 n_1)$ time. From $M$, we compute $M^{-1}$ in $O(\beta^\omega)$ time.  Then, to compute $v' = P_\Gamma\cdot v$, we need three matrix to vector multiplications that can be done in $O(\beta n_1)$, $O(\beta^2)$, and $O(\beta n_1)$ time respectively. The total running time for computing $v'$ is $O(\beta^2 n_1 + \beta^\omega)$.

It remains to show the Loewner bound of the lemma for $P_\Gamma$.
By Lemma~\ref{lem:max_eig_bound_row_col}-(i), we have
\[
\left|\left|\left[\begin{array}{c}
   p_1^T \\
   \vdots \\
   p_\beta^T
\end{array}
\right]\cdot
v\right|\right|/\|v\| \leq \sqrt{\beta} \cdot p_{\max},
\]
for any $v$. For any $u$, by Lemma~\ref{lem:eigen_lower_based_abs_values}, we have
\[
\frac{\|M^{-1}u\|}{\|u\|} =
\frac{\|u'\|}{\|M u'\|} \leq
\beta^{\frac{\beta}{2}+1}\cdot M^n_{\max}.
\]
But all elements of $M$ are inner products of $p_i$'s and $\gamma_i$'s.  So, $M_{\max} \leq p_{\max}\cdot\gamma_{\max}$.
Finally, for any $x$, by Lemma~\ref{lem:max_eig_bound_row_col}-(ii),
\[
\|[\gamma_1, \ldots, \gamma_\beta]x\|/\|x\| \leq \sqrt{\beta}\cdot\gamma_{\max}.
\]
Putting everything together,
\begin{align*}
\frac{\|P_{\Gamma} v\|}{\|v\|} &\leq (\sqrt{\beta}\cdot\gamma_{\max})(\beta^{\frac{\beta}{2}+1} (p_{\max}\cdot\gamma_{\max})^{\beta})(\sqrt{\beta} \cdot p_{\max}) =
\beta^{\frac{\beta+3}{2}}\cdot (p_{\max}\cdot\gamma_{\max})^{\beta+1},
\end{align*}
for any $v$. Thus
$
\|P_{\Gamma}\| \leq \beta^{(\beta+3)/2}(p_{\max}\cdot\gamma_{\max})^{\beta+1}.
$
Equivalently, by Lemma~\ref{lem:matrix_norm_to_loewner_bound},
\[
-\beta^{(\beta+3)/2}\cdot (p_{\max}\cdot\gamma_{\max})^{\beta+1}\cdot I
\preceq
-P_{\Gamma} \preceq \beta^{(\beta+3)/2}\cdot (p_{\max}\cdot\gamma_{\max})^{\beta+1}\cdot I,
\]
which is stronger than the lemma statement.
\end{proof}

We obtain the following corollary of Lemma~\ref{lem:PGamma} by using the homology basis $\Gamma$ of Lemma \ref{lem:tamal_hom_basis} and the cohomology basis $P$ obtained by applying the operator $C$ of Section \ref{subsec:Black_algorithm} to $\Gamma$. As each element of $\Gamma$ is a $\{-1, 0, +1\}$ vector, then $\gamma_{\max} = \sqrt{n_1}$. This observation and Lemma \ref{lem:bound_Cgamma} imply $p_{\max} = \alpha\cdot n_1^{5/2}n_{2}^{4}/\lambda_{\min}(X)$, where $\alpha$ is a constant and $\lambda_{\min}(X)$ is the smallest non-zero eigenvalue of $L^{up}_1(X)$.

\begin{restatable}{corollary}{PGamma}
\label{cor:PGamma}
Let $X, K$ as defined.
Let $\Gamma = \{\gamma_1, \ldots, \gamma_\beta\}$ be the basis of Lemma~\ref{lem:tamal_hom_basis},
and let $P_{\Gamma}$ be the operator that for any cycle $\alpha$ returns the unique linear combination of the cycles of $\Gamma$ that is homologous to $\alpha$.
We have
\[
    (1 - P_{\Gamma})(1 - P_{\Gamma})^T \preceq
    \eps \cdot I,
\]
for $\eps = (n_1 n_2/\lambda_{\min}(X))^{c\cdot \beta}$, where $\lambda_{\min}(X)$ is the smallest non-zero eigenvalue of $L^{up}_1(X)$ and $c$ is a constant.
Further,
for any vector $v$, $P_{\Gamma}\cdot v$ can be computed in $O(\beta^2 n_1 + \beta^\omega)$ time.
\end{restatable}

\subsection{The boundary projection operator}
Now, we are ready to prove the main lemma of this section that describes a projection operator into the space of the boundary cycles.

\boundaryprojection*

\begin{proof}
Let $\delta = {\varepsilon}/({2\cdot n_1^4\cdot(n_1n_2/\lambda_{\min})^{c\beta}})$, where $c$ is the constant of Corollary~\ref{cor:PGamma}.
Let $\widetilde{\Pi}_{cbd}(\delta)$ and $\widetilde\Pi_{hr}(\delta)$ be the operators of Lemma~\ref{lem:graph_projection} and Lemma~\ref{lem:harmonic_projection} respectively. Thus,
\begin{align}
    &\Pi_{cbd} - \delta\cdot I \preceq \widetilde{\Pi}_{cbd}(\delta)\preceq
    \Pi_{cbd} + \delta\cdot I \label{eqn:pi_cbd_bound}
    \\
    &\Pi_{hr} - \delta\cdot I \preceq \widetilde{\Pi}_{hr}(\delta)\preceq \Pi_{hr} + \delta\cdot I \label{eqn:pi_harm_bound}
\end{align}
We define
\[
\widetilde{\Pi}_{bd}(\eps) =
(I - P_{\Gamma})(I- P_T)
\left(I-\widetilde{\Pi}_{cbd}(\delta)-\widetilde{\Pi}_{hr}(\delta)\right)
(I- P_T)^T(I - P_{\Gamma})^T.
\]
From Equations (\ref{eqn:pi_cbd_bound}) and (\ref{eqn:pi_harm_bound}), we have
\[
\Pi_{bd}- 2\delta I \preceq
I - \widetilde{\Pi}_{cbd}(\delta) - \widetilde{\Pi}_{hr}(\delta) \preceq
\Pi_{bd}+ 2\delta I.
\]
Multiplying by $(I - P_{\Gamma})(I- P_T)$ on both sides, and using Lemma~\ref{lem:PTPGammaProp}-(i) and the definition of $\widetilde{\Pi}_{bd}(\eps)$,
\begin{align}
\label{eqn:PiBD_bound}
\Pi_{bd}- 2\delta (I - P_{\Gamma})(I- P_T) (I - P_{T})^T(I- P_\Gamma)^T &\preceq
\widetilde{\Pi}_{bd}(\eps)
\preceq \nonumber\\
&\Pi_{bd} + 2\delta (I - P_{\Gamma})(I- P_T) (I - P_{T})^T(I- P_\Gamma)^T.
\end{align}
Next, we bound $(I - P_{\Gamma})(I- P_T)  (I- P_T)^T(I - P_{\Gamma})^T$.
\begin{align*}
(I - P_{\Gamma})(I- P_T) (I- P_T)^T(I - P_{\Gamma})^T &\preceq n_1^4 (1 - P_{\Gamma})(1 - P_{\Gamma})^T &&\text{(Lemma~\ref{lem:PT})} \\
&\preceq
 (n_1n_2/\lambda_{\min})^{c\cdot \beta}\cdot  n_1^4 I
&&\text{(Corollary~\ref{cor:PGamma})}
\end{align*}
Multiplying by $\Pi_{bd}$ on both sides we obtain,
\[
\Pi_{bd} (I - P_{\Gamma})(I - P_T)(1-P_T)^T(1 - P_{\Gamma})^T \Pi_{bd} \preceq
(n_1n_2/\lambda_{\min})^{c\cdot \beta}\cdot  n_1^4 \cdot \Pi_{bd}.
\]
We obtain $\Pi_{bd} (I - P_{\Gamma})(I - P_T)(1-P_T)^T(1 - P_{\Gamma})^T \Pi_{bd} = (I - P_{\Gamma})(I - P_T)(1-P_T)^T(1 - P_{\Gamma})^T $, by applying Lemma~\ref{lem:PTPGammaProp}-(ii) twice.
It follows that
\[
(I - P_{\Gamma})(I - P_T)(1-P_T)^T(1 - P_{\Gamma})^T \preceq
(n_1n_2/\lambda_{\min})^{c\cdot \beta}\cdot  n_1^4 \cdot \Pi_{bd}.
\]
Substituting this new bound in (\ref{eqn:PiBD_bound}),
\begin{align*}
\Pi_{bd}- 2\delta\cdot (n_1n_2/\lambda_{\min})^{c\cdot \beta}\cdot  n_1^4  \cdot \Pi_{bd} \preceq
\widetilde{\Pi}_{bd}(\eps)
\preceq \nonumber
\Pi_{bd} + 2\delta\cdot (n_1n_2/\lambda_{\min})^{c\cdot \beta}\cdot  n_1^4 \cdot  \Pi_{bd}.
\end{align*}
We obtain the bound of the lemma by replacing $\delta$ with its value $\delta = {\varepsilon}/({2\cdot n_1^4\cdot(n_1n_2/\lambda_{\min})^{c\beta}})$.

Now we analyze the runtime of applying $\widetilde\Pi_{cbd}$. By Lemma~\ref{lem:PT} and Corollary~\ref{cor:PGamma}, we can apply $(I-P_{T})$, $(I-P_{\Gamma})$, and their transposes to a vector in $O(\beta^2 n_1 + \beta^\omega)$ time.
By Lemma~\ref{lem:graph_projection}, we can apply $\widetilde\Pi_{cbd}(\delta)$ to a vector in time
\[
\tilde{O}(n\log n \log (n/\delta)) = \tilde{O}(\beta n \log n \log (n / (\lambda_{\min}\eps))).
\]
Finally, by Lemma~\ref{lem:harmonic_projection}, $\widetilde\Pi_{hr}(\delta)$ can be applied to a vector in time
\[
\tilde{O}\left(
    \beta^2 \cdot n \cdot \log n\cdot  \log\left(\frac{n}{\lambda_{\min}\cdot\delta}\right)
\right) = \tilde{O}\left(
     \beta^3 \cdot n \cdot \log n \cdot \log\left(\frac{n}{\lambda_{\min}\cdot\eps}\right)
\right);
\]
this is the bottleneck of all the running times, so it determines the running time of applying $\widetilde\Pi_{bd}(\eps)$.
\end{proof}

\section{Laplacian solver}
\label{sec:solver}
We end this paper with the description of our Laplacian solver.
Recall $L_1 = L_1^{down} + L_1^{up} = \partial_1^T\partial_1 + \partial_2\partial_2^T$. Since $\im(L_1^{down})$ is orthogonal to $\im(L_1^{up})$, we have $(L_1)^+ = (L_1^{down})^+ + (L_1^{up})^+$.  Therefore, we can approximate $(L_1^{down})^+$ and $(L_1^{up})^+$ separately.  Cohen et al.~describe an algorithm to approximate $(L_1^{down})$ for all graphs, see their Lemma 4.2.  The following lemma is a restatement of the same lemma by Black et al.

\begin{lemma}[Black et al.~\cite{CollapsibleUniverseBlack22}, Lemma 4.8]
\label{lem:down_solver}
For a simplicial complex $K$ and $\varepsilon>0$, there is a map $DownLaplacianSolver(\varepsilon)$ such that
\[
(1-\varepsilon)(L_1^{down}[K])^+ \preceq DownLapSolver(K, \varepsilon) \preceq (L_1^{down}[K])^+.
\]
Further, for $x \in C_1$, $DownLapSolver(K, \varepsilon)\cdot x$ can be computed in $\T{O}(n_1\log^2 (n_1/\varepsilon))$ time.
\end{lemma}
For solving the up-Laplacian, Black et al.~rely on the following lemma.

\approximationbbt

Lacking the operator for projecting into the boundary space, Black et al.~project into the cycle space instead, which happens to be the boundary space if $K$ has trivial homology.  But now, we can use the operator of Lemma~\ref{lem:boundary_form_cbd_harm} to solve for the $1$-Laplacian of $K$ with arbitrary homology, hence Theorem~\ref{thm:main_solver_thm}.

The running time of Black et al.~depends on the condition number of $L_1^{up}$ within the space of boundary cycles, which is equal to the maximum eigenvalue of $L_1^{up}$ divided by its minimum nonzero eigenvalue.  The following lemma allows us to reduce this dependence to only the minimum nonzero eigenvalue.
\begin{lemma}[Black and Maxwell~\cite{black_resistance}, Lemma 40]
\label{lem:bounded_lambda_max}
Let $L^{up}_d$ be the up $d$th Laplacian of a simplicial complex with $n_{d+1}$ $(d+1)$-simplices. We have
$\lambda_{\max}(L^{up}_d) \leq n_{d+1}(d+1)$.
\end{lemma}
\begin{proof}
Recall $\lambda_{\max}(L^{up}_d) = \sigma^2_{\max}(\partial_d)$ and $\sigma_{\max}(\partial_d) = \max\{\|\partial_d x\|/\|x\| : x\in C_{d+1}(K)\}$. For a vector $x\in C_{d+1}(K)$, let $\widehat x_i$ be the vector whose $i$th coordinate is $x[i]$, and whose other coordinates are zero. Since $x = \sum{\widehat x_i}$ we have
\[
\partial_d x = \sum{\partial_d \widehat x_i} \Rightarrow
\|\partial_d x\| \leq \sum{\|\partial_d \widehat x_i\|} = \sum{\sqrt{(d+1)(x[i])^2}}
=\sqrt{d+1} \|x\|_1 \leq \sqrt{n_{d+1}(d+1)} \|x\|,
\]
hence $\|\partial_d x\|/\|x\| \leq \sqrt{n_d(d+1)}$ as desired.
\end{proof}
We are now ready to prove the main theorem of this section.

\solverthm*
\begin{proof}
By Lemma~\ref{lem:down_solver}, we have a down Laplacian solver for $K$.  We use Lemma~\ref{lem:black_etal_bbt} with $B$ being $\partial_2$ to obtain an up Laplacian solver.
To that end, we need an operator $U$ that for each $b\in \im(\partial_2)$ returns $x$ such that $\partial_2 x = b$. Lemma 11 of Black et al.~describes such an operator $U$. Also, we need $\Pi_{\ker^\perp(\partial_2)}$ that is equal to $\Pi_{\im(\partial_2^T)}$ the space of coboundary $2$-chains.  Black et al.~show that this coboundary space is dual to a cycle space of a graph, even if $K$ has nontrivial homology. Thus, $\T\Pi_{cyc}$ from Lemma~\ref{lem:graph_projection} can be used.  Finally, we need $\T\Pi_{\im(\partial_2)}$, which we obtain from Lemma~\ref{lem:boundary_form_cbd_harm}.

By Lemma~\ref{lem:black_etal_bbt}, $\T\Pi_{cyc}(\eps')$ and $\T\Pi_{bd}(\eps')$ with $\eps' = \Omega(\eps/\kappa(L_1^{up}(K)))$.  Thus, for any vector $x$, $\T\Pi_{cyc}(\eps') x$ can be computed in $\T O(n\log n\log(n_1\kappa(L_1^{up}(K))/\eps))$ time by Lemma~\ref{lem:graph_projection}.  Also, by Lemma \ref{lem:boundary_form_cbd_harm}, $\T\Pi_{bd}(\eps') x$ can be computed in
\begin{align*}
\T O\left(
    \beta^3 \cdot n \cdot \log n \cdot \log\frac{n}{\lambda_{\min}(L^{up}(X))\cdot\eps'}
\right) &= \T O\left(
    \beta^3 \cdot n \cdot \log n \cdot \log\frac{n\cdot \kappa(L_1^{up}(K))}{\lambda_{\min}(L_1^{up}(X))\cdot\eps}
\right) \\ &= \T O\left(
    \beta^3 \cdot n \cdot \log n \cdot \log\frac{n}{\lambda_{\min}(L_1^{up}(K))\cdot \lambda_{\min}(L_1^{up}(X))\cdot\eps}
\right)
\end{align*}
time. The last equality follows from Lemma~\ref{lem:bounded_lambda_max}.
\end{proof}

\end{document}